\documentclass[a4paper, 11pt]{amsart}

\usepackage{amsmath}
\usepackage{amsthm}
\usepackage{amsfonts}
\usepackage{amssymb}
\usepackage{amscd}
\usepackage[dvips]{graphicx, color}

\newtheorem{thm}{Theorem}[section]
\newtheorem{cor}[thm]{Corollary}

\newtheorem{lem}[thm]{Lemma}
\newtheorem{prop}[thm]{Proposition}
\newtheorem{rem}[thm]{\normalfont Remark}
\newtheorem{exam}[thm]{\normalfont Example}

\numberwithin{equation}{section}

\begin{document}
\title[$q$-deformed integers derived from pairs of coprime integers]{\bf $q$-deformed integers derived from pairs of coprime integers and its applications}
\author{
{
Michihisa Wakui}
}
\address{Department of Mathematics, Faculty of Engineering Science, 
Kansai University, Suita-shi, Osaka 564-8680, Japan}
\email{wakui@kansai-u.ac.jp}


\thanks{This paper is dedicated in memory of Toshie Takata san.}

\maketitle 
\begin{abstract}
In connection with cluster algebras, snake graphs and $q$-integers, Kyungyong Lee and Ralf Schiffler recently found a formula for computing the (normalized) Jones polynomials of rational links in terms of continued fraction expansion of rational numbers. 
Sophie Morier-Genoud and Valentin Ovsienko introduced $q$-deformed continued fractions, and showed that by using them each coefficient of the normalized Jones polynomial counted quiver representations of type $A_n$. 
In this paper we introduce $q$-deformed integers defined by pairs of coprime integers, which are motivated by the denominators and the numerators of their $q$-deformed continued fractions, and give an efficient algorithm for computing the (normalized) Jones polynomials of rational links. 
Various properties of $q$-integers defined by pairs of coprime integers are
investigated and shown its applications. 
\end{abstract}

\footnote[0]{2020 Mathematics Subject Classification. 57K10, 57K14, 11A55, 05A30, 13F60}

\baselineskip 16pt 

For computation of the Jones polynomials for rational links, various formulas have already been known \cite{BBL, Kanenobu, Kogiso-Wakui_Proc, Kogiso-Wakui, Kogiso-Wakui_OJM, LLS, LeeSchiffler, M-GO2, Nakabo1, Nakabo2, QYAQ, Yamada-Proceeding}. 
Among them Lee and Schiffler \cite{LeeSchiffler} study on the Jones polynomials of rational links from connection with cluster algebras and snake graphs, and give formulas for computation of them based on continued fractions. 
Morier-Genoud and Ovsienko \cite{M-GO2} introduce a notion of $q$-deformed continued fractions and show that 
the normalized Jones polynomials whose constant term is $1$ can be computed from their $q$-deformed rational numbers. 
Kogiso and the author \cite{Kogiso-Wakui_Proc, Kogiso-Wakui, Kogiso-Wakui_OJM} define the Kauffman bracket polynomials for Conway-Coxeter friezes of zigzag type, and show that the friezes correspond to the rational links up to isomorphism. 
\par 
In the present paper we introduce a new algorithm for computation of the Jones polynomials of rational links and give some applications of them. 
Our approach is based on $q$-Farey sums by the negative continued fraction expansions of rational numbers, which are introduced by Morier-Genoud and Ovsienko \cite{M-GO2}, and 
the formula computing writhes of rational link diagrams given by Nagai and Terashima~\cite{Nagai_Terashima}. 
\par 
For a rational number $\alpha$ the normalized Jones polynomial $J_{\alpha }(q)$ of the corresponding rational link to $\alpha$ can be simply computed in comparison with the original Jones polynomial $V_{\alpha }(t)$ under the substitution of $t=-q^{-1}$. 
In fact, the normalized Jones polynomial can be inductively computed by $q$-Farey sums, which relate with hyperbolic geometry, cluster algebras, Euler-Ptolemy-Pl\"{u}cker relation, and so on. 
According to the original definition by Morier-Genoud and Ovsienko, for computing the $q$-Farey sum of $q$-rational numbers $[\frac{x}{a}]_q$ and $[\frac{y}{b}]_q$, the negative continued fraction expansion of the Farey sum $\frac{x}{a}\sharp \frac{y}{b}$ is required. 
This gives rise to an interesting phenomenon as follows. 
For two rational numbers whose denominators are common, namely $r$, the denominators of their $q$-rational numbers does not coincide while they are $q$-deformations of the same integer $r$. 
To analyze this fact, we introduce $q$-deformations $(a, b)_q$ of a positive integer $k$ associated with pairs $(a, b)$ of coprime and positive integers such that $a+b=k$. 
In the present paper various properties and equations on $(a, b)_q$ are investigated. 
In particular, it is shown that 
the $q$-rational number of the Farey sum $\frac{x}{a}\sharp \frac{y}{b}$ is given by $\frac{(x,y)_q}{(a, b)_q}$.  
This means that the $q$-rational number $[\frac{x}{a}\sharp \frac{y}{b}]_q$ can be obtained by computing from the numerators $x, y$ and the denominators $a, b$, separately. 
It is also given a formula of computing the normalized Jones polynomial $J_{\alpha }(q)$ from $(a, b)_q$. 
Though $(a, b)_q$ actually coincides with the numerator of $q$-rational number $[\frac{a+b}{b}]_q$, it would be expected that $(a, b)_q$ is useful and significant to study $q$-rational numbers and related areas, for instance, construction and analysis of $q$-deformations of Conway-Coxeter friezes. 
\par 
When we wish to recover the original Jones polynomial from the normalized one, 
the most complicated factor is writhes. 
In \cite{LeeSchiffler} the formula for computing writhes of rational link diagrams by using even fraction expansions. 
However, the even fraction expansion of a rational number does not always exist 
though it can be transformed into one having even fraction expansion. 
Thanks to Nagai and Terashima's formula~\cite{Nagai_Terashima}, 
Kogiso and the author~\cite{Kogiso-Wakui_OJM} gave a purely combinatorial formula of the writhe $\mathrm{wr}(\alpha )$ based on the regular continued fraction expansion of $\alpha$.  
In the present paper we derive its negative version. 
\par 
This paper is organized as follows. 
In Section 1 
the definitions of rational links and their normalized Jones polynomials are recalled. 
A useful method of recovering the original Jones polynomials from the normalized ones is described. 
In Section 2 
the above statement is proved. 
Moreover, based on negative continued fraction expansions, a combinatorial formula for computing writhes of rational link diagrams is given.  
In Section 3
we introduce $q$-deformations of integers derived from pairs of coprime integers. 
Various properties of them are investigated. 
In particular, we give a formula of computing 
$q$-Farey sums of $q$-rational numbers defined by Morier-Genoud and Ovsienko \cite{M-GO2} from our $q$-integers derived from pairs of coprime integers. 
In the final section 
it is shown that the normalized Jones polynomials can be computed by $q$-integers derived from pairs of coprime integers. 
As an application of it, 
it is also shown that our $q$-integers actually are given by numerators of $q$-deformed rational numbers.  
So, the coefficients of the $q$-integers have a quiver theoretical meaning given in \cite{M-GO2}. 
\par 
We use the following notations: For an $x\in \mathbb{R}$, we set 
$\lceil x \rceil :=\min \{ \ n\in \mathbb{Z}\ | \ x\leq n\ \} ,\ 
\lfloor x\rfloor :=\max \{ \ n\in \mathbb{Z}\ | \ n\leq x \ \}$. 
We note that if $\alpha \not\in \mathbb{Z}$, then 
$\lfloor \alpha +1\rfloor =\lceil \alpha \rceil $. 
We refer to \cite{Cromwell, Murasugi} et al for basic theory of knots and links.

\section{Rational links and their normalized Jones polynomials}
\par 
Rational links are one of fundamental classes in knots and links. 
In some context they are called two-bridge links.  
They are defined by continued fraction expansions of rational numbers as follows. 
\par 
A rational number $\alpha >0$ can be represented by the following continued fraction: 
$$
[a_1, a_2, \ldots , a_n]:=a_1+\dfrac{1}{\vbox to 18pt{ }a_2+\dfrac{1}{\vbox to 18pt{ }\ddots +\dfrac{1}{\vbox to 18pt{ }a_{n-1}+\dfrac{1}{a_n}}}}, 
$$
where 
$a_1\geq 0$ and $a_2, \ldots , a_n>0$ are integers. 
This representation is unique if  parity of $n$ is specified. 
\par 
Let $\alpha \in \mathbb{Q}$ with $0<\alpha <1$, and write in the form 
$\alpha =[0, a_1, \ldots , a_n]$, where $n$ is odd. 
The link in the $3$-sphere $\mathbb{S}^3$ determined by the diagram $D(\alpha )$ below is called a rational link. 
\begin{equation*}
D( \alpha ) =\ \raisebox{-0.6cm}{\includegraphics[height=1.5cm]{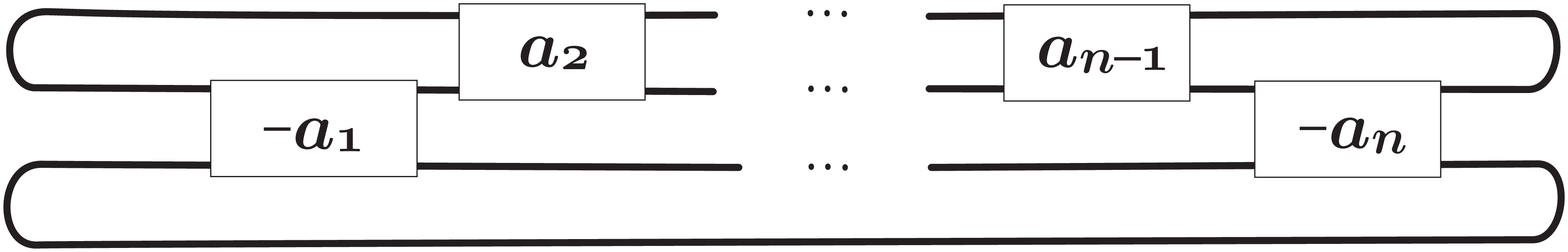}}\, 
\end{equation*}
\noindent 
where 
$$\raisebox{-0.3cm}{\includegraphics[height=0.8cm]{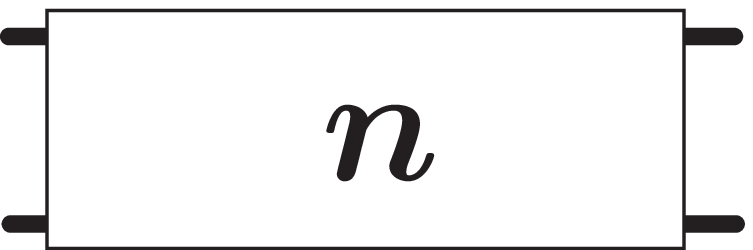}}\ \ =\ 
\begin{cases}
\raisebox{-0.5cm}{\includegraphics[height=1.2cm]{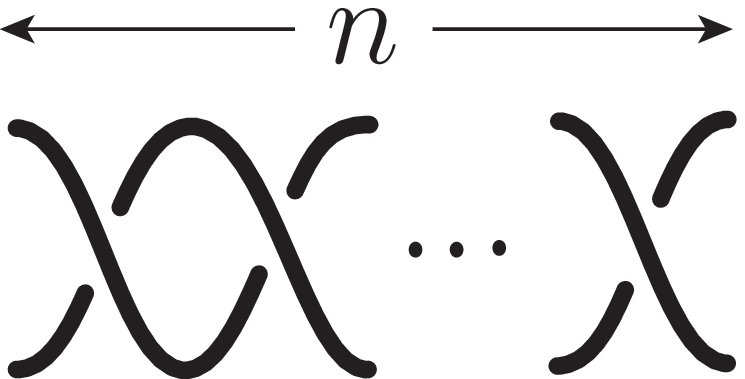}}\ & \raisebox{-0.2cm}{if $n\geq 0$,}\\[0.5cm]  
\raisebox{-0.5cm}{\includegraphics[height=1.2cm]{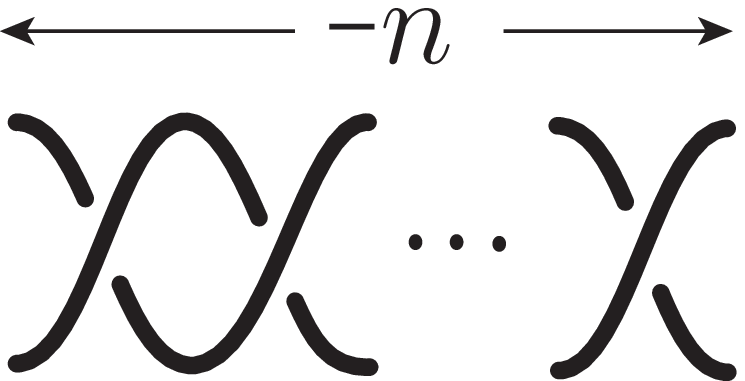}}\ & \raisebox{-0.2cm}{if $n< 0$.}
\end{cases} $$

\par 
If $\alpha >1$, then $D(\alpha )$ is defined by $D(\alpha ):=D(\alpha ^{-1})$, and 
if $\alpha =1$, then $D(\alpha )=\raisebox{-0.5cm}{\includegraphics[height=1.2cm]{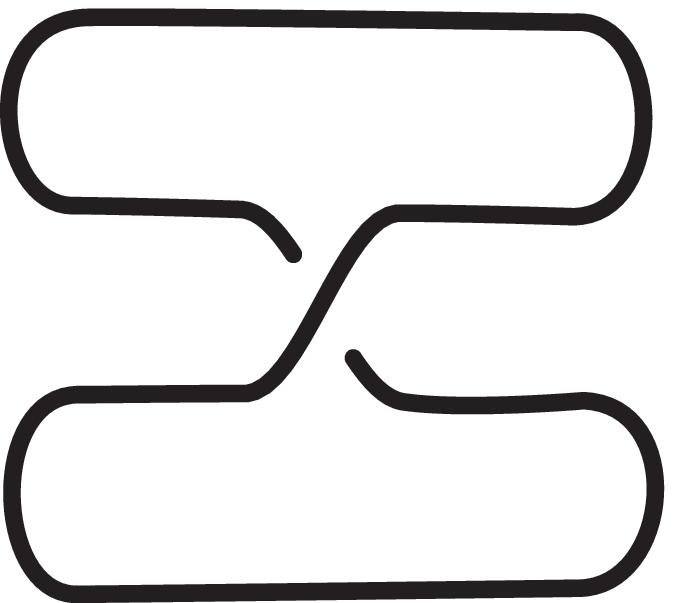}}$.  
According as the denominator of $\alpha$ is odd or even, $D( \alpha )$ is a knot or a two-component link, respectively. 
Orientations for $D(\alpha )$ are given as follows: 

\begin{list}{}{\labelsep=0.1cm \labelwidth=0.8cm \leftmargin=0.8cm}
\item[$\bullet$] 
If the denominator of $\alpha $ is odd, then 
\begin{equation*}
\ \raisebox{-0.4cm}{\includegraphics[height=1.5cm]{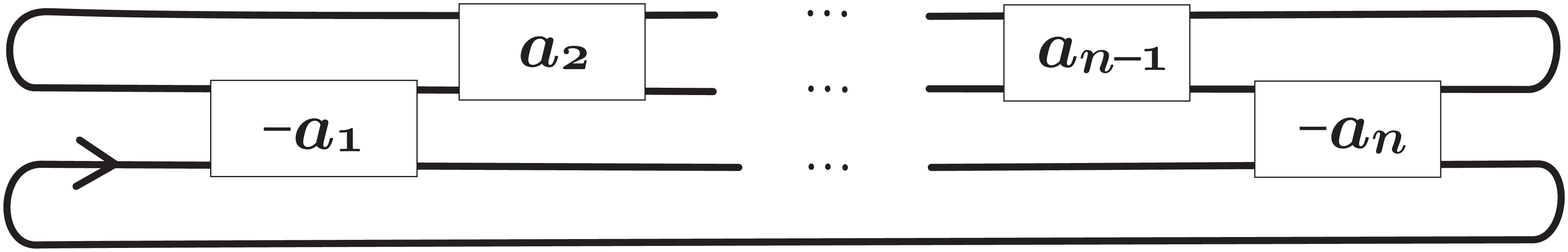}}\ . 
\end{equation*}
\item[$\bullet$] 
If the denominator of $\alpha $ is even, then 
\begin{equation*}
\ \raisebox{-0.4cm}{\includegraphics[height=1.5cm]{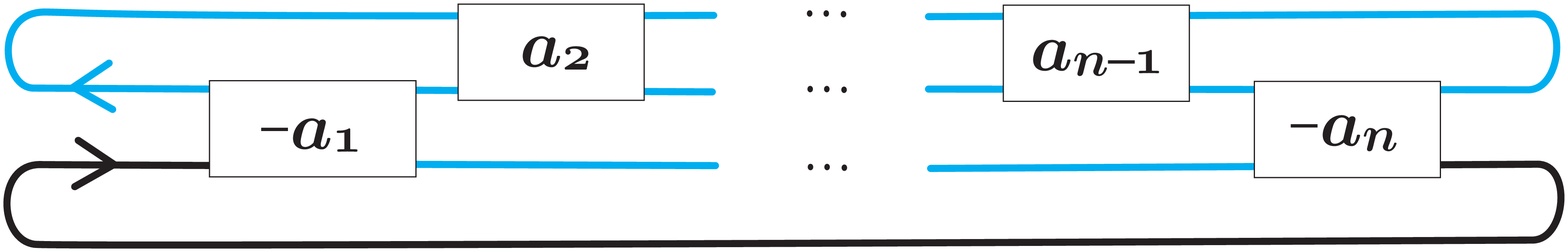}}\ . 
\end{equation*}

In this case there is another orientation for $D( \alpha )$ which is given by 

\begin{equation*}
\ \raisebox{-0.6cm}{\includegraphics[height=1.5cm]{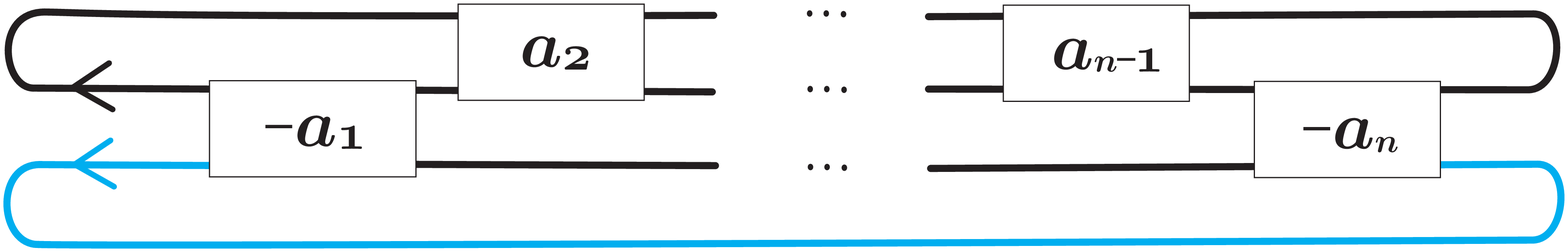}}\ . 
\end{equation*}

We denote the oriented link diagram by $D_{+-}( \alpha )$. 
\end{list} 

\par 
For an $\alpha \in \mathbb{Q}$ the oriented rational link determined by $D(\alpha )$ is denoted by $L(\alpha )$. 
\par 
For a link diagram $D$, the Kauffman bracket polynomial $\langle D\rangle $ is defined. 
It is a regular isotopy invariant of $D$ and takes a value in the Laurent polynomial ring $\mathbb{Z}[A, A^{-1}]$. 
The Jones polynomial \cite{Jones, Kau-Topology} is an isotopy invariant for an oriented link $L$ in $\mathbb{S}^3$, which is valued in $\mathbb{Z}[t^{\pm \frac{1}{2}}]$, and 
computed from the Kauffman bracket polynomial by the formula
\begin{equation}
V_{L}(t)=(-A^3)^{-\text{wr}(D)}\langle D\rangle \bigl|_{A=t^{-\frac{1}{4}}}, 
\end{equation}
\noindent 
where $\text{wr}(D)$ is the writhe of an oriented link diagram $D$ of $L$, which is the sum of signs $\pm 1$ running over all crossings when we   correspond a sign to each crossing as in Figure~\ref{fig1}: 

\begin{figure}[htbp]
\centering 
\includegraphics[height=1.5cm]{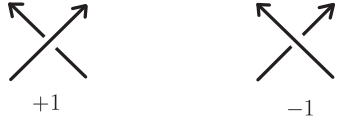}
\caption{signs of crossings}\label{fig1}
\end{figure}

For a rational number $\alpha $, consider the following normalization $J_{\alpha }(q)$ of $V_{\alpha }(t):=V_{L( \alpha )}(t)$~\cite{LeeSchiffler}:  
\begin{equation}
J_{\alpha }(q):=\pm t^{-h}V_{\alpha }(t)|_{t=-q^{-1}}, 
\end{equation}
where $\pm t^{h}$ is the leading term of $V_{\alpha }(t)$. 
This indicates the normalization such that the constant term is $1$ as a polynomial in $q$. 

\begin{exam}
Since $V_{\frac{12}{5}}(t)=-t^{\frac{3}{2}}+t^{\frac{1}{2}}-3t^{-\frac{1}{2}}+2t^{-\frac{3}{2}}-2t^{-\frac{5}{2}}+2t^{-\frac{7}{2}}-t^{-\frac{9}{2}}$, 
$J_{\frac{12}{5}}(q)=1+q+3q^2+2q^3+2q^4+2q^5+q^6$. 
\end{exam}

By Lee and Schiffler~\cite{LeeSchiffler}, it is known that the Jones polynomial  $V_{\alpha }(t)$ can be recovered from 
$J_{\alpha }(q)$. 
Their formula is given by even continued fractions. 
Unfortunately $\alpha$ often does not have an even continued fraction expansion though it can be always transformed into one having even continued fraction expansion \cite[Proposition 2.3]{LeeSchiffler}. 
In this paper we give an alternative formula for it in terms of a negative continued fraction expansion of $\alpha$. 
Here, by a negative continued fraction we means an expression of a rational number as in the form 
\begin{equation}\label{eq1-3}
c_1-\dfrac{1}{\vbox to 18pt{ }c_2-\dfrac{1}{\vbox to 18pt{ }\ddots -\dfrac{1}{\vbox to 18pt{ }c_{l-1}-\dfrac{1}{c_l}}}}, 
\end{equation}
where $c_1, \ldots , c_l$ are integers which are greater than or equal to $2$. 
We denote the fraction \eqref{eq1-3} by $[c_1, \ldots , c_l]^-$, which is denoted by $[\kern-0.1em [ c_1, \ldots , c_l ]\kern-0.1em ]$ in \cite{M-GO2}. 
Any rational number $\alpha \ (>1)$ can be uniquely represented by a negative continued fraction. 
Though the uniqueness does not hold, 
the fraction \eqref{eq1-3} has a meaning even if $c_i=1$ for some $i$.  

\begin{thm}\label{1-3}
For a rational number $\alpha >1$, 
\begin{equation}\label{eq1-4}
V_{\alpha}(t)=(-t)^{-\frac{3}{4}\mathrm{wr}(\alpha )-\frac{3}{4}l+\frac{7}{4}l^{\prime}}J_{\alpha}(-t^{-1}), 
\end{equation}
where $\mathrm{wr}(\alpha )=\mathrm{wr}(D(\alpha ))$ and $l^{\prime}=\sum_{j=1}^l(c_j-2)+1$ for $\alpha =[c_1, \ldots , c_l]^-$ with $c_j\geq 2$ for all $j=1, \ldots , l$. 
\end{thm}

The theorem would be proved in the next section. 
Since the writhe $\mathrm{wr}(\alpha )$ can be purely computed in a combinatorial way by negative continued fraction expansion as shown in the next section, 
the Jones polynomial $V_{\alpha}(t)$ is recovered from the normalized Jones polynomial $J_{\alpha}(q)$ in a purely combinatorial way. 

\par \medskip 
\section{Computation of the Jones polynomials for rational links based on negative continued fractions}

Irreducible fractions $\frac{x}{a}, \frac{y}{b}$ are said to be   
\textit{Farey neighbors} if $ay-bx=1$. 
In this paper the following conventions are assumed: 
\begin{enumerate}
\item[$\bullet$] $\infty =\frac{1}{0}$ is regarded as an irreducible fraction. 
\item[$\bullet$] For any irreducible fraction $\frac{x}{a}$, it is always $q\geq 0$. 
\end{enumerate}

If $\frac{x}{a}, \frac{y}{b}$ are Farey neighbors, then $\frac{x}{a}\sharp \frac{y}{b}:=\frac{x+y}{a+b}$, called the \textit{Farey sum}, is also irreducible. 
On Farey neighbors the following lemmas hold: 

\begin{lem}\label{2-1}
\begin{enumerate}
\item[$(1)$] Any non-negative rational number can be obtained from $\frac{0}{1}$ and $\frac{1}{0}$ applying  $\sharp$ in finitely many times. 
\item[$(2)$] For any $\alpha \in \mathbb{Q}$, there are uniquely Farey neighbors $\frac{x}{a}, \frac{y}{b}$ such that $\alpha =\frac{x}{a}\sharp \frac{y}{b}$. The pair $(\frac{x}{a}, \frac{y}{b})$ is called the parents of $\alpha$. 
\end{enumerate}
\end{lem} 

For a proof of the above lemma, see \cite[Theorem 3.9]{Aigner} or \cite[Lemma 3.5]{Kogiso-Wakui_Proc}. 

\begin{lem}\label{2-2}
If $\frac{x}{a}, \frac{y}{b} >0$ are  Farey neighbors, then 
\begin{equation}
\left\lceil \dfrac{x+y}{a+b} \right\rceil =\begin{cases}
\left\lceil \dfrac{x}{a} \right\rceil & (\text{if}\ a>1)\\[0.25cm]  
\left\lceil \dfrac{x}{a} \right\rceil +1& (\text{if}\ a=1)
\end{cases}
\ =\left\lfloor \dfrac{x}{a} +1\right\rfloor  . 
\end{equation}
\end{lem} 
\begin{proof}
If $a=1$, then  $y=bx+1$. Since 
$\frac{x+y}{a+b}=x+\frac{1}{b+1}$, 
we have 
$\lceil \frac{x+y}{a+b} \rceil =x+1=\lceil \frac{x}{a} \rceil +1$. 
\par 
Next, consider the case $a>1$. 
If $b=1$, then $x=ay-1$, and hence 
$\frac{x+y}{a+b}=y-\frac{1}{a+1}$. 
Thus, $\lceil \frac{x+y}{a+b} \rceil =y=\lceil \frac{x}{a} \rceil $. 
If $b>1$, then $x, y$ are expressed as 
$x=ma+r,\ y=nb+s$ 
for some $r\ (0<r<a),\ s\ (0<s<b)$. 
\par 
If $m\geq n$, then 
\begin{equation*}
\dfrac{x+y}{a+b}=n+\dfrac{(m-n)a+(r+s)}{a+b}. \tag*{$(\ast )$}
\end{equation*}
Since 
$ay-bx=1$, we have $ab(n-m)+sa-rb=(nb+s)a-(ma+r)b=1$. 
Thus 
$(m-n)a=\frac{sa-rb-1}{b}$, and it follows that the right-hand side of $(\ast )$ is less than 
$n+\frac{s}{b}<n+1$. 
This implies that 
$\lceil \frac{x+y}{a+b} \rceil =n+1=\lceil \frac{y}{b} \rceil $. 
On the other hand, since $\frac{x}{a}<\frac{y}{b}$, we have 
$\lceil \frac{x}{a}\rceil \leq \lceil \frac{y}{b} \rceil $ 
while $\lceil \frac{x}{a} \rceil \geq \lceil \frac{y}{b} \rceil $ by 
$m\geq n$. 
Thus 
$m+1=\lceil \frac{x}{a}\rceil =\lceil \frac{y}{b} \rceil =n+1$. 
This means that if $m\geq n$, then $m=n$, and 
$\lceil \frac{x+y}{a+b} \rceil =\lceil \frac{y}{b} \rceil =\lceil \frac{x}{a}\rceil $. 
\par 
If $n>m$, then 
$\frac{x+y}{a+b}=m+\frac{(n-m)b+(r+s)}{a+b}=m+\frac{1+r(a+b)}{a(a+b)}
<m+1$. 
Therefore, $\lceil \frac{x+y}{a+b} \rceil =m+1=\lceil \frac{x}{a} \rceil $. 
\end{proof} 

From Lemma~\ref{2-1} we have a binary tree with extra vertices $\frac{0}{1}, \frac{1}{0}$ and two (dotted) edges, 
whose vertices are the non-negative rational numbers, depicted as in Figure~\ref{fig2}. 
We call the tree the (extended) \textit{Stern-Brocot tree}, which is essentially the same with the Farey tessellation appeared in hyperbolic geometry. 
We note that any non-negative rational number is appeared only once in the Stern-Brocot tree as a vertex. 

\begin{figure}[htbp]
\centering\includegraphics[width=11cm]{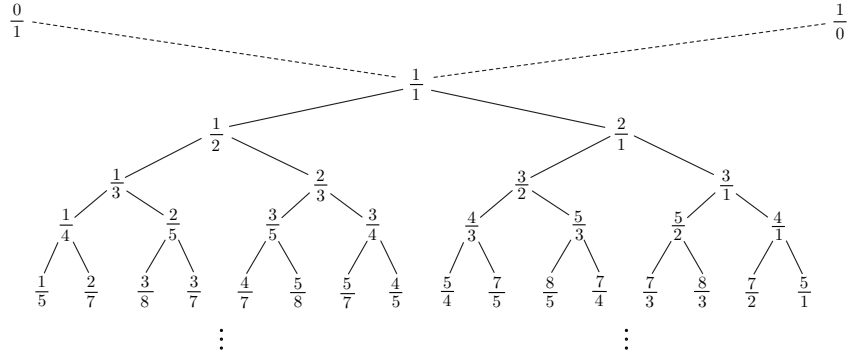}
\caption{the (extended) Stern-Brocot tree}\label{fig2}
\end{figure}

Based on the Stern-Brocot tree, for each positive rational number $\alpha $ one can find a triangle $\text{YAT}(\alpha )$ which is shaped as \raisebox{-0.6cm}{\includegraphics[width=1.8cm]{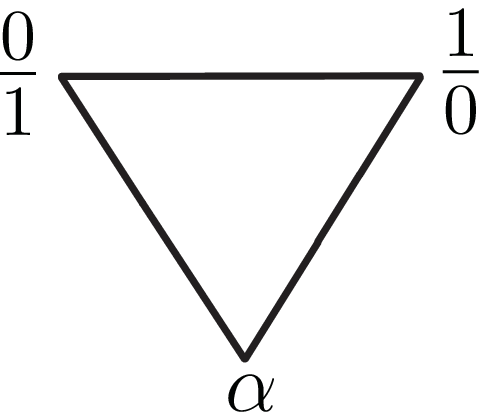}}\ and decomposed into small triangles according to the following rules: 

\begin{enumerate}\itemindent=1cm 
\item[(YAT1)] All vertices are on two oblique edges. 
\item[(YAT2)] For any small triangle such as \kern-0.2em 
\raisebox{-0.6cm}{\includegraphics[width=2.2cm]{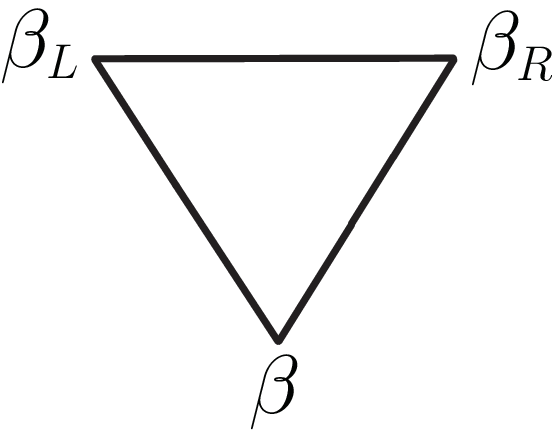}}, 
$\beta =\beta _L\sharp \beta _R$. 
\end{enumerate}

The triangle $\text{YAT}(\alpha )$ is called the \textit{ancestor triangle} of $\alpha$, and 
each small triangle in $\text{YAT}(\alpha )$ is called a \textit{fundamental triangle}. 
This concept is introduced by Shuji Yamada~\cite{Yamada-Proceeding} to study of the Jones polynomials of two-bridge links. 
The same concept is also introduced by Hatcher and Ortel~\cite{HO} from the more geometrical point of view. 

\begin{exam}
Let $\alpha =\frac{7}{4}$. Then 
$\alpha =\frac{5}{3}\sharp \frac{2}{1}$,\ 
$\frac{5}{3}=\frac{3}{2}\sharp \frac{2}{1}$,\ 
$\frac{3}{2}=\frac{1}{1}\sharp \frac{2}{1}$,\  
$\frac{2}{1}=\frac{1}{1}\sharp \frac{1}{0}$,\ 
$\frac{1}{1}=\frac{0}{1}\sharp \frac{1}{0}$. 
Thus,  
$\mathrm{YAT}(\frac{7}{4})$ consists of the blue parts in Figure~\ref{fig3}. 

\begin{figure}[htbp]
\centering\includegraphics[width=7cm]{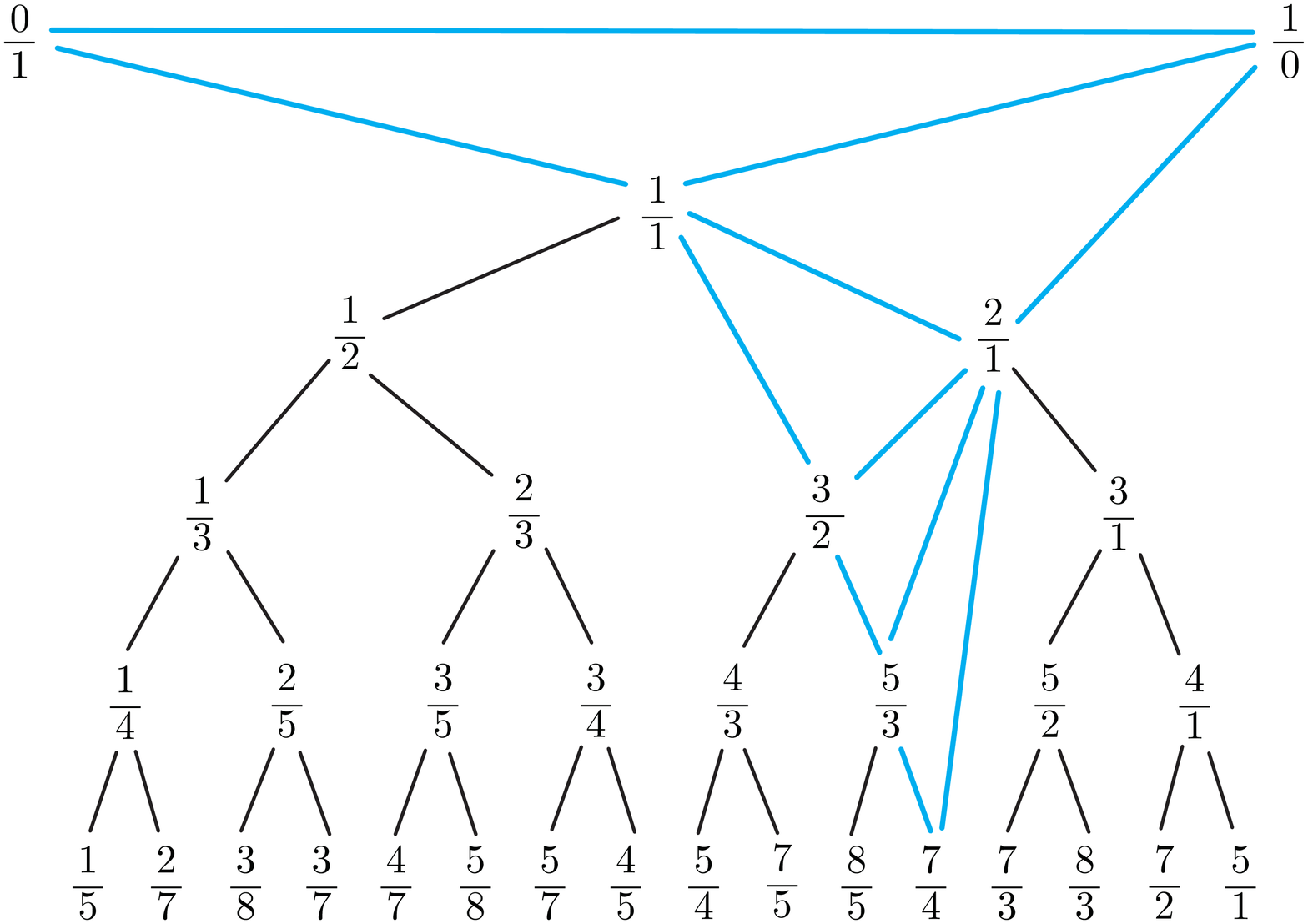}
\hspace{0.5cm} \raisebox{2cm}{$\Longrightarrow$ } \hspace{0.5cm} 
\raisebox{2cm}{$\mathrm{YAT}\bigl( \frac{7}{4} \bigr)  =$} \hspace{0.1cm} 
\includegraphics[width=3.5cm]{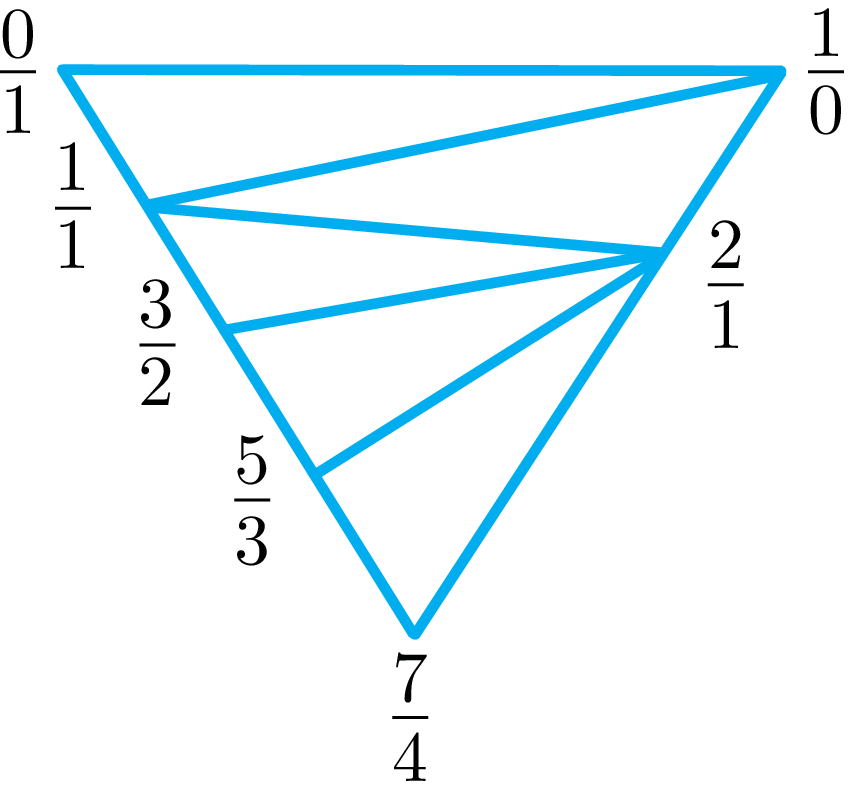}
\caption{the ancestor triangle of $\frac{7}{4}$}\label{fig3}
\end{figure}

\end{exam} 

Let $\alpha \in (0,1)\cap \mathbb{Q}$. We then define a Laurent polynomial $\langle \varGamma _{\alpha }\rangle \in \mathbb{Z}[t, t^{-1}]$ as follows. 
First of all, for each fundamental triangle, we set $-t^{-1}$ and $-t$ on the left and right oblique edges, respectively. 
For an upward path $\gamma$ from $\alpha $ to $\frac{0}{1}$ or $\frac{1}{1}$ 
we compute the product $W(\gamma )$ of $-t^{\pm 1}$'s on all edges passed by $\gamma$.  
Finally, we set $\langle \varGamma _{\alpha }\rangle =\sum _{\gamma } W(\gamma )$ under replacing $t=A^{-4}$. 
By previous joint works with Kogiso,  
it is shown that $\langle \varGamma _{\alpha }\rangle $ coincides with the Jones polynomial $V_{\alpha}(t)$ up to multiplying $\pm t^{\frac{3}{4}k}$ for some $k\in \mathbb{Z}$. 
More precisely we have the following: 

\begin{prop}[{\cite[Proposition 4.3]{Kogiso-Wakui_OJM}}]\label{2-4}
For an $\alpha \in (0,1)\cap \mathbb{Q}$, 
\begin{equation}\label{eq2-1}
V_{\alpha }(t)=(-A^{-3})^{\widetilde{\mathrm{wt}}(\alpha )}\langle \varGamma_{\alpha } \rangle |_{A=t^{-\frac{1}{4}}},
\end{equation}
where $\widetilde{\mathrm{wt}}(\alpha )=-\mathrm{wr}(\alpha )-\mathrm{wt}(\alpha )$, and 
$\mathrm{wt}(\alpha )=-\sum_{j=1}^l(c_j-2)+l-1$ for $\alpha ^{-1}=[c_1, \ldots , c_l]^-$ with $c_j\geq 2$ for all $j$. 
\end{prop}

\begin{rem}
Note that the above replacement $t=A^{-4}$ is inverse of that in the papers \cite{Kogiso-Wakui, Kogiso-Wakui_OJM}. 
\end{rem} 

Let $\alpha \in (0,1)\cap \mathbb{Q}$, and define two paths $\gamma _L$ and $\gamma _R$ as in Figure~\ref{fig4}, 
that is, $\gamma _L$ is the path from $\alpha$ to $\frac{0}{1}$ along the left oblique edge, and 
$\gamma _R$ is the path from $\alpha$ to $\frac{1}{1}$ along the right oblique edge. 
We note that $\frac{1}{1}$ is always on the right oblique side in $\text{YAT}(\alpha )$. 

\begin{figure}[hbtp]
\centering\includegraphics[width=3cm]{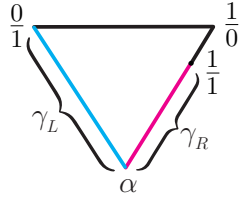}
\caption{the paths $\gamma _L$ and $\gamma _R$}\label{fig4}
\end{figure}

By counting the number of edges on the left and right oblique sides, respectively, 
$W(\gamma _L)$ and $W(\gamma _R)$ are given as in the following lemma. 

\begin{lem}\label{2-6}
For an $\alpha \in (0,1)\cap \mathbb{Q}$ with expression $\alpha^{-1}=[c_1, \ldots , c_l]^-$ 
\begin{equation}\label{eq2-3}
W(\gamma _L)=(-t^{-1})^l, \qquad 
W(\gamma _R)=(-t)^{l^{\prime}}, 
\end{equation}
where 
$l^{\prime}:= \sum_{j=1}^l (c_j-2) +1$. 
They are the lowest and highest term of $\langle \varGamma _{\alpha }\rangle $, respectively. 
\end{lem} 

\par \medskip 
By using the lemma Theorem~\ref{1-3} can be proved as follows. 

\par \medskip \noindent 
{\bf Proof of Theorem~\ref{1-3}.} \ 
By Proposition~\ref{2-4} 
$V_{\alpha }(t)=V_{\alpha ^{-1}}(t)=(-t^{\frac{3}{4}})^{\widetilde{\mathrm{wt}}(\alpha ^{-1})}\langle \varGamma_{\alpha ^{-1}}\rangle $. 
So, by Lemma~\ref{2-6} the leading term of $V_{\alpha }(t)$ is given by 
$(-t^{\frac{3}{4}})^{\widetilde{\mathrm{wt}}(\alpha ^{-1})}W(\gamma_R)=(-t)^{\frac{3}{4}\widetilde{\mathrm{wt}}(\alpha ^{-1})+l^{\prime}}$. 
Since 
$\widetilde{\mathrm{wt}}(\alpha ^{-1})
=-\mathrm{wr}(\alpha ^{-1})-\mathrm{wt}(\alpha ^{-1}) 
=-\mathrm{wr}(\alpha )+l^{\prime}-l$, 
we have 
$(-t^{\frac{3}{4}})^{\widetilde{\mathrm{wt}}(\alpha ^{-1})}W(\gamma_R)
=(-t)^{-\frac{3}{4}\mathrm{wr}(\alpha )-\frac{3}{4}l+\frac{7}{4}l^{\prime}}$. 
This implies the equation \eqref{eq1-4}. 
\qed 

\par \medskip 
Let us explain that some relationship between negative continued fractions and ancestor triangles. 
For an $\alpha \in (0,1)\cap \mathbb{Q}$, 
we express its inverse as 
$\alpha ^{-1}=[c_1, \ldots , c_l]^-$. Then 
the vertices in the both 
oblique sides of $\text{YAT}(\alpha )$ are 
as in Figure~\ref{fig5}.   

\begin{figure}[hbtp]
\centering\includegraphics[height=4cm]{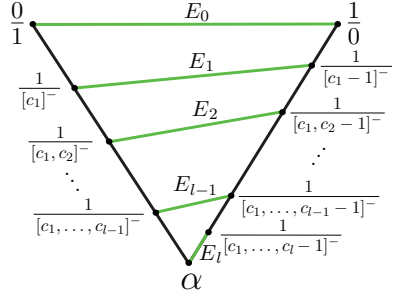}
\caption{vertices of $\text{YAT}(\alpha )$ and negative continued fractions}\label{fig5}
\end{figure}

For each $j$ 
the number of  the fundamental 
triangles in  the region enclosed by $E_{j-1}, E_j$, 
the both oblique sides is $c_1$ for $j=1$, and is $(c_j-1)$ for the others $j$. 
We note that if $c_l=2$, then end points of $E_{l-1}$ and $E_l$ coincide since $[c_1, \ldots , c_l-1]^-=[c_1, \ldots , c_{l-1}-1]^-$. 

\begin{exam}
If $\alpha =\frac{7}{11}$, then 
$\alpha ^{-1}=\frac{11}{7}=[2, 3, 2, 2]^-$ and $\mathrm{YAT}\bigl( \frac{7}{11}\bigr)$ is given by Figure~\ref{fig6}. 
\end{exam} 

\begin{figure}[hbtp]
\centering\includegraphics[height=4cm]{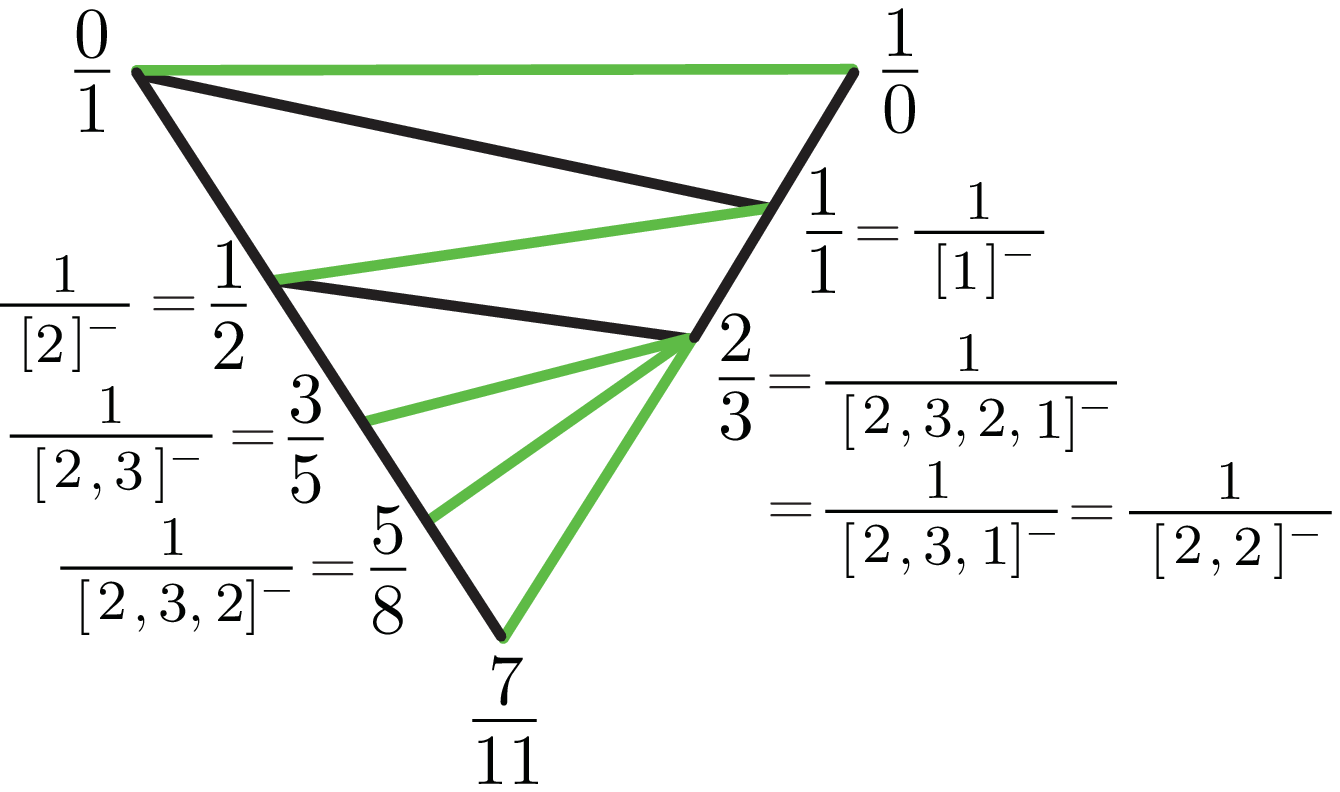}
\caption{vertices of $\text{YAT}\bigl( \frac{7}{11}\bigr)$}\label{fig6}
\end{figure}

By a direct application of Nagai and Terashima's results \cite{Nagai_Terashima} 
one can obtain a recursive formula  \cite[Theorem 5.2]{Kogiso-Wakui_OJM} to compute the writhe $\text{wr}(\alpha )$ from the regular continued fraction expansion of $\alpha $. 
By using the conversion between regular and negative continued fraction expansions given in \cite{Hirzebruch, HirzebruchZagier, M-GO2}, 
we have another recursive formula to compute $\text{wr}(\alpha )$. 
To explain the formula let us recall the definition of a Seifert path, which is introduced by Nagai and Terashima \cite{Nagai_Terashima}. 
\par 
The rational numbers are classified into three types such as $\frac{1}{1}, \frac{1}{0}, \frac{0}{1}$-types. 
A rational number $\frac{x}{a}$ is called $\frac{1}{1}$-type  if $x\equiv 1,\ a\equiv 1\ (\text{mod}\ 2)$, and $\frac{1}{0}$-type and $\frac{0}{1}$-type are similarly defined. 

Let $\alpha $ be a rational number in $(0,1)$. 
All of types appear in the vertices in each fundamental triangle of $\text{YAT}(\alpha )$. 
A \textit{Seifert path} of $\alpha $ is a downward path in $\text{YAT}(\alpha )$, which is started from $\frac{1}{0}$ to $\alpha $ satisfying the following condition: 
The end points of any edge in the path consist of $\frac{1}{1}$- and $\frac{1}{0}$-types, or consist of $\frac{1}{0}$- and $\frac{0}{1}$-types. 
If the denominator of $\alpha $ is odd, then a Seifert path is unique. 
We denote it by $\gamma _{\alpha }$. 
If the denominator of $\alpha $ is even, namely $\alpha $ is of type $\frac{1}{0}$, then there are exactly two Seifert paths. 
In this case we denote by $\gamma _{\alpha }$ the Seifert path whose vertices consist of $\frac{1}{0}$- and $\frac{0}{1}$-types, 
and denote by $\gamma _{\alpha }^{\prime}$ the remaining Seifert path. 
\par 
Following \cite{Nagai_Terashima} let us explain how to define a sign $t_{\alpha }(\Delta )$ for each fundamental triangle $\Delta $ in $\text{YAT}(\alpha )$. 
We consider the successive sequence of the fundamental triangles of $\text{YAT}(\alpha )$ whose initial term is the fundamental triangle with the edge joining the vertices $\frac{1}{0}$ and $\frac{0}{1}$. 
If $\Delta $ is the initial triangle, then  we set 
$$t_{\alpha }(\Delta ):=\begin{cases}
1 & \text{if $\alpha$ is of $\frac{1}{1}$-type}, \\[0.1cm]  
-1 & \text{otherwise}. 
\end{cases}$$
Assume that for the previous triangle $\Delta_-$ of $\Delta$, the sign $t_{\alpha }(\Delta_-)$ is defined. 
Then $t_{\alpha }(\Delta )$ is defined by
$$t_{\alpha }(\Delta )=\begin{cases}
t_{\alpha }(\Delta _-) & \text{if $\gamma _{\alpha}$ does not pass through between $\Delta _-$ and $\Delta $},\\[0.1cm]  
-t_{\alpha }(\Delta _-) & \text{otherwise}.
\end{cases}$$

\begin{thm}[{\bf Nagai and Terashima \cite[Theorem 4.4]{Nagai_Terashima}}]\label{2-8}
For an $\alpha \in (0,1)\cap \mathbb{Q}$ the writhe of $D(\alpha )$ is given by 
\begin{equation}\label{eq2-4}
-\mathrm{wr}(\alpha ) =\sum\limits_{\text{the fundamental triangles $\Delta $ in $\text{YAT}(\alpha )$}} t_{\alpha }(\Delta ). 
\end{equation}
\end{thm} 

If $\alpha \in (0,1)\cap \mathbb{Q}$ is $\frac{1}{0}$-type, then one can also define another sign $t_{\alpha }^{\prime}(\Delta )$ for each fundamental triangle $\Delta $ in $\text{YAT}(\alpha )$ by using $\gamma _{\alpha }^{\prime}$ as follows. 
For the initial triangle $\Delta $  we set 
$t_{\alpha }^{\prime}(\Delta ):=-1$. 
Assume that for the previous triangle $\Delta_-$ of $\Delta$, the sign $t_{\alpha }^{\prime}(\Delta_-)$ is defined. 
Then $t_{\alpha }^{\prime}(\Delta )$ is defined by
$$t_{\alpha }^{\prime}(\Delta )=\begin{cases}
t_{\alpha }^{\prime}(\Delta _-) & \text{if $\gamma _{\alpha}^{\prime}$ does not pass through between $\Delta _-$ and $\Delta $},\\[0.1cm]  
-t_{\alpha }^{\prime}(\Delta _-) & \text{otherwise}.
\end{cases}$$

Then by the same way of proof of Theorem~\ref{2-8}, it is shown that the equation
\begin{equation}\label{eq2-5}
-\text{wr}_{+-}(\alpha ) =\sum\limits_{\text{the fundamental triangles $\Delta $ in $\text{YAT}(\alpha )$}}t_{\alpha }^{\prime}(\Delta )
\end{equation}
\par \noindent 
holds, where $\mathrm{wr}_{+-}(\alpha )=\mathrm{wr}(D_{+-}(\alpha ))$. 
By rewriting the formulas \eqref{eq2-4} and \eqref{eq2-5} we have the following proposition. 

\begin{figure}[hbtp]
\centering\includegraphics[height=4.3cm]{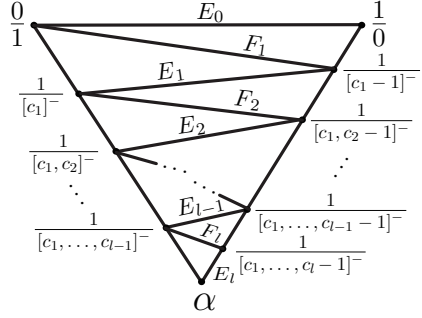}
\caption{vertices and edges of $\text{YAT}(\alpha )$}\label{fig7}
\end{figure}

\begin{prop}\label{2-9}
Let $\alpha \in \mathbb{Q}\cap (0,1)$ with expression 
$\alpha ^{-1}=[c_1, \ldots , c_l]^-$. 
Let $E_0$ be the edge between $\frac{0}{1}$ and $\frac{1}{0}$, and 
for each $j=1, \ldots , l$ let $E_j, F_j$ be the edges between $\frac{1}{[c_1, \ldots , c_j]^-}$ and $\frac{1}{[c_1, \ldots , c_j-1]^-}$, and between $\frac{1}{[c_1, \ldots , c_{j-1}]^-}$ and $\frac{1}{[c_1, \ldots , c_j-1]^-}$, respectively. 
Denoted by $z_j$ is the number of times that the Seifert path $\gamma_{\alpha}$ crosses $\mathrm{YAT}(\alpha )$ 
while it starts from $\frac{1}{0}$ and comes first to one of the edge points of $E_j$, 
where if $\gamma_{\alpha}$ passes through $E_0$, then we count it. 
\par 
Let $t_{\alpha }([E_{j-1}, E_j])$ be the recursively defined number as follows: 
$$t_{\alpha }([E_0, E_1]):=
\begin{cases}
-c_1 & \text{if $z_1=0$},\\[0.1cm]   
c_1 & \text{if $z_1=1$},\\[0.1cm]   
c_1-2 & \text{if $z_1=2$}, 
\end{cases}$$
and for $j\geq 2$ 
$$t_{\alpha }([E_{j-1}, E_j]):=
\begin{cases}
(-1)^{z_j-1}(c_j-1) & \text{if $\gamma_{\alpha}$ does not pass through $F_j$},\\[0.1cm]  
(-1)^{z_j}(c_j-3)  &  \text{otherwise}.
\end{cases}$$
Then we have 
\begin{equation}\label{eq2-6}
\mathrm{wr}(\alpha )=\sum\limits_{j=1}^l t_{\alpha }([E_{j-1}, E_j]). 
\end{equation}
Furthermore, if $\alpha $ is $\frac{1}{0}$-type, then $\mathrm{wr}_{+-}(\alpha )$ can be computed as follows. 
Denoted by $z_j^{\prime}$ is the number of times that the Seifert path $\gamma_{\alpha}^{\prime}$ crosses $\mathrm{YAT}(\alpha )$ 
while it starts from $\frac{1}{0}$ and comes first to one of the edge points of $E_j$. 
Let $t_{\alpha }^{\prime}([E_{j-1}, E_j])$ be the recursively defined number as follows: 
$t_{\alpha }^{\prime}([E_0, E_1]):=-c_1$, and for $j\geq 2$
$$t_{\alpha }^{\prime}([E_{j-1}, E_j]):=
(-1)^{z_j^{\prime}-1}(c_j-1).$$
Then we have 
\begin{equation}\label{eq2-7}
\mathrm{wr}_{+-}(\alpha ) =\sum\limits_{j=1}^l t_{\alpha }^{\prime}([E_{j-1}, E_j]). 
\end{equation}
\end{prop}
\begin{proof}
To prove \eqref{eq2-6}, for each $j=1, \ldots , l$, 
let $\gamma _{\alpha}^{(j)}$ be the subpath of $\gamma _{\alpha}$ 
from the edge point of $E_{j-1}$ first reached  
to the edge point of $E_j$ first reached. 
We also denote by $[E_{j-1}, E_j]$ the quadrilateral enclosed by $E_{j-1}, E_j$ and the both oblique sides. 
By Theorem~\ref{2-8} it is sufficient to show that $-t_{\alpha }([E_{j-1}, E_j])$ coincides with the sum of signs of the fundamental triangles in $[E_{j-1}, E_j]$. 
For $j>1$ it can be verified by dividing into the following four cases: 
\begin{enumerate}
\item[(1)] $z_j=z_{j-1}$,  
\item[(2)] $z_j=z_{j-1}+1$ and  $\gamma _{\alpha}^{(j)}$ passes through $E_{j-1}$, 
\item[(3)] $z_j=z_{j-1}+1$ and  $\gamma _{\alpha}^{(j)}$ passes through $F_j$, 
\item[(4)] $z_j=z_{j-1}+2$. 
\end{enumerate}

In the case of Part (1), $\gamma _{\alpha}^{(j)}$ does not pass through $E_{j-1}$ and $F_j$. 
So, the sum of signs of the fundamental triangles in $[E_{j-1}, E_j]$ is $(-1)^{z_{j-1}}(c_j-1)$, and therefore 
multiplying it by $-1$ we have $(-1)^{z_{j-1}-1}(c_j-1)=(-1)^{z_j-1}(c_j-1)=t_{\alpha }([E_{j-1}, E_j])$. 
In all other cases, by the same argument, we see that the desired equation holds. 
\par 
For $j=1$ the desired equation can be also verified by dividing into the three cases where 
$z_j=z_{j-1}$, $z_j=z_{j-1}+1$ and $z_j=z_{j-1}+2$. 
\par 
The equation \eqref{eq2-7} can be obtained by a quite similar argument. 
\end{proof} 

\begin{rem} 
By \cite[Lemma 4.1(3)]{Kogiso-Wakui_OJM}, we also have 
\begin{equation}\label{eq2-8}
\mathrm{wr}_{+-}\bigl(\alpha ) =-\mathrm{wr}( 1-\alpha ) . 
\end{equation}
\end{rem}

\par \medskip 
By Proposition~\ref{2-8} we have: 

\par\medskip 
\noindent 
\begin{thm}\label{2-11}
Let $\alpha \in \mathbb{Q}\cap (0,1)$, and express its inverse as 
$\alpha ^{-1}=[c_1, \ldots , c_l]^-$. 
Define $E_j, F_j$ and  $z_j$ as in Proposition~\ref{2-9} and 
set $\alpha _j:=\frac{1}{[c_1, \ldots , c_j]^-},\ \beta_j:=\frac{1}{[c_1, \ldots , c_j-1]^-}$,  
$z_j(\alpha _j):=z_j$ for $j=1, \ldots , l$. 
Then the writhe $\mathrm{wr}(\alpha )$ can be recursively computed as follows: 
\begin{enumerate}
\item[$(1)$] If $l=1$, that is, $\alpha =\frac{1}{c_1}$, then 
$\mathrm{wr}(\alpha )=(-1)^{c_1}c_1$. 
In addition, if $c_1$ is even, then $\mathrm{wr}_{+-}(\alpha )=-c_1$.  
\item[$(2)$] If $l=2$, that is, $\frac{1}{\alpha }=[c_1, c_2]^-$, then 
$\mathrm{wr}(\alpha )=(-1)^{(c_1-1)c_2}c_1+(-1)^{c_1}c_2-1$. 
In addition, if both $c_1, c_2$ are odd, then $\mathrm{wr}_{+-}(\alpha )=-c_1+c_2-1$.  
\item[$(3)$] Let $l\geq 3$ and set 
$$\epsilon (c_l):=
\begin{cases}
+ & \text{if $c_l$ is even},\\ 
- & \text{if $c_l$ is odd}, 
\end{cases}$$
and $\mathrm{wr}_{++}(\alpha ):=\mathrm{wr}(\alpha )$. Then  
\begin{enumerate}
\item[$(i)$] if $\alpha _{l-1}$ is $\frac{1}{0}$-type and $\alpha $ is $\frac{1}{1}$-type, then 
$$\mathrm{wr}(\alpha ) =\mathrm{wr}_{+-}(\alpha _{l-1}) +(-1)^{z_{l-1}(\alpha _{l-1})+\frac{1+(-1)^{c_l}}{2}+1}(c_l-1),$$
\item[$(ii)$] if $\alpha _{l-1}$ is $\frac{0}{1}$-type and $\alpha $ is $\frac{1}{1}$-type, then 
$$\mathrm{wr}(\alpha )
=\mathrm{wr}_{+\epsilon (c_l)}(\alpha _{l-2}) +(-1)^{z_{l-2}(\alpha _{l-2})+\frac{1+(-1)^{c_{l-1}}}{2}+\frac{1-(-1)^{c_l}}{2}}(-c_{l-1}+c_l+2),$$
\item[$(iii)$] if $\alpha _{l-1}$ is $\frac{1}{0}$-type and $\alpha $ is $\frac{0}{1}$-type, then 
$$\mathrm{wr}(\alpha ) =\mathrm{wr}(\alpha _{l-1}) +(-1)^{z_{l-1}(\alpha _{l-1})+\frac{1+(-1)^{c_l}}{2}}(c_l-1),$$
\item[$(iv)$] $\alpha _{l-1}$ is $\frac{1}{1}$-type and $\alpha $ is $\frac{0}{1}$-type, then 
$$\mathrm{wr}(\alpha )=\mathrm{wr}(\alpha _{l-2}) +(-1)^{z_{l-2}(\alpha _{l-2})+\frac{1+(-1)^{c_{l-1}}}{2}}(-c_{l-1}+c_l+2),$$
\item[$(v)$] if $\alpha _{l-1}$ is $\frac{0}{1}$-type and $\alpha $ is $\frac{1}{0}$-type, then
\begin{align*}
\mathrm{wr}(\alpha ) &=\mathrm{wr}(\alpha _{l-1}) +(-1)^{z_{l-1}(\alpha _{l-1})+\frac{1+(-1)^{c_l}}{2}}(c_l-1),\\ 
\mathrm{wr}_{+-}(\alpha ) &=
\mathrm{wr}_{+\epsilon (c_l-1)}(\alpha _{l-2}) +(-1)^{z_{l-2}(\alpha _{l-2})+\frac{1+(-1)^{c_{l-1}}}{2}+\frac{1+(-1)^{c_l}}{2}}(c_{l-1}+c_l-2) , 
\end{align*}
\item[$(vi)$] if $\alpha _{l-1}$ is $\frac{1}{1}$-type and $\alpha $ is $\frac{1}{0}$-type, then 
\begin{align*}
\mathrm{wr}(\alpha ) &=\mathrm{wr}(\alpha _{l-2}) +(-1)^{z_{l-2}(\alpha _{l-2})+\frac{1+(-1)^{c_{l-1}}}{2}}(-c_{l-1}+c_l+2),\\ 
\mathrm{wr}_{+-}(\alpha ) &=\mathrm{wr}(\alpha _{l-1}) +(-1)^{z_{l-2}(\alpha _{l-2})+\frac{1+(-1)^{c_{l-1}}}{2}+1}(c_l-1). 
\end{align*}
\end{enumerate}
\end{enumerate}
\end{thm} 
\begin{proof}
It can be easily verified in the case where $n=1, 2$. 
So, let $n\geq 3$. 
Here, we only give the proof of (ii) since other cases are shown by a quite similar argument. 
In the case of (ii), 
the types of $\beta_{l-1}$ and $\alpha_{l-2}$ are $\frac{1}{0}$-type or $\frac{1}{1}$-type, and 
they are different.  Thus, $\gamma_{\alpha}$ is expressed as 
$$\gamma_{\alpha}=\begin{cases}
\gamma _{\alpha _{l-2}}\ast (\alpha _{l-2}\to \beta _{l-1}\to \cdots \to \alpha ) & \text{if $\beta_{l-1}$ is $\frac{1}{0}$-type}, \\ 
\gamma _{\alpha _{l-2}}^{\prime}\ast (\alpha _{l-2}\to \beta _{l-1}\to \cdots \to \alpha )& \text{if $\beta_{l-1}$ is $\frac{1}{1}$-type}. 
\end{cases}$$ 

We see that 
$$z_l(\alpha )=z_{l-1}(\alpha )=\begin{cases}
z_{l-2}(\alpha )+1& \text{if $\beta_{l-1}$ is $\frac{1}{0}$-type and $\beta_{l-2}$ is $\frac{0}{1}$-type}, \\ 
z_{l-2}(\alpha )+2& \text{if $\beta_{l-1}$ is $\frac{1}{0}$-type and $\beta_{l-2}$ is $\frac{1}{0}$-type},\\ 
z_{l-2}(\alpha )+1& \text{if $\beta_{l-1}$ is $\frac{1}{1}$-type and $\beta_{l-2}$ is $\frac{0}{1}$-type},\\ 
z_{l-2}(\alpha )+2& \text{if $\beta_{l-1}$ is $\frac{1}{1}$-type and $\beta_{l-2}$ is $\frac{1}{1}$-type},
\end{cases}$$
and 
$$z_{l-2}(\alpha )=\begin{cases}
z_{l-2}(\alpha _{l-2}) & \text{if $\beta_{l-1}$ is $\frac{1}{0}$-type},\\ 
z_{l-2}(\alpha _{l-2})-1 & \text{if $\beta_{l-1}$ is $\frac{1}{1}$-type}. 
\end{cases}
$$ 
The number of fundamental triangles in the triangle whose vertices are $\alpha _{l-1}, \beta _{l-1}, \alpha $ is $c_l-1$. 
So, $\beta_{l-1}$ is $\frac{1}{1}$-type if and only if $c_l$ is odd. 
Therefore, $z_{l-2}(\alpha )$ can be written as 
$$z_{l-2}(\alpha )=z_{l-2}(\alpha _{l-2})-\dfrac{1-(-1)^{c_l}}{2}.$$
Similarly, considering the number of fundamental triangles in the region enclosed by $E_{l-2}, E_{l-1}$ and the both oblique sides, we see that $z_{l-1}(\alpha )$ can be written as 
$$z_{l-1}(\alpha )=z_{l-2}(\alpha )+1+\dfrac{1+(-1)^{c_{l-1}}}{2}.$$
Substituting the above two equations into 
\begin{align*}
\text{wr}(\alpha )
&=\text{wr}_{+\epsilon (c_l)}(\alpha _{l-2}) +t([E_{l-2}, E_{l-1}])+t([E_{l-1}, E_l]) \\ 
&=\text{wr}_{+\epsilon (c_l)}(\alpha _{l-2}) +(-1)^{z_{l-1}(\alpha )}(c_{l-1}-3)+(-1)^{z_l(\alpha )-1}(c_l-1), 
\end{align*}
we have the equation in (ii). 
\end{proof} 

\par \medskip 
\begin{rem}\label{2-12}
The type of the rational number $\alpha =[c_1, \ldots ,c _l]^-$ with $c_j\geq 2$ for all $j$ can be inductively determined  as follows: 
Let $n(\alpha )$ and $d(\alpha )$ be the parities of the numerator and the denominator of $\alpha$, respectively. 
We set $\alpha_j=[c_1, \ldots  , c_j]^-$ for $j=1, \ldots , l$. Then we have 
\begin{enumerate}
\item[$(1)$] $n(\alpha_1)=\frac{1-(-1)^{c_1}}{2},\ d(\alpha_1)=1$, 
\item[$(2)$] $n(\alpha_2)=\frac{1+(-1)^{c_1c_2}}{2},\ d(\alpha_2)=\frac{1-(-1)^{c_2}}{2}$, 
\item[$(3)$] for $j\geq 3$ 
\begin{align*}
n(\alpha _j)&=n(\alpha _{j-1})\frac{1-(-1)^{c_j}}{2}+n(\alpha _{j-2}),\\ 
d(\alpha _j)&=d(\alpha _{j-1})\frac{1-(-1)^{c_j}}{2}+d(\alpha _{j-2}). 
\end{align*}
\end{enumerate}
\end{rem}

\begin{exam}
(1) Let us consider the case $\alpha =\frac{13}{21}$. 
Since $\frac{21}{13}=[2,3,3,2]^-$ and 
$\alpha _3=\frac{1}{[2,3,3]^-}=\frac{8}{13},\ 
\alpha _2=\frac{1}{[2,3]^-}=\frac{3}{5},\ 
\alpha _1=\frac{1}{[2]^-}=\frac{1}{2}$, we will apply the formula of Theorem~\ref{2-11}(3)(ii). 
Then we have 
$\mathrm{wr}\bigl( \frac{13}{21}\bigr) =\mathrm{wr}\bigl( \frac{3}{5}\bigr) +(-1)^{z_2(\alpha_2)}$. 
Here, 
$z_2(\alpha_2)=1,\ \mathrm{wr}\bigl( \frac{3}{5}\bigr) =0$. Thus 
$\mathrm{wr}\bigl( \frac{13}{21}\bigr) =-1$. 
Since $l=4,\ l^{\prime}=3$ and 
$$J_{\frac{21}{13}}(q) =1+3q+3q^2+4q^3+4q^4+3q^5+2q^2+q^7,$$
by Theorem~\ref{1-3} we have 
$$V_{\frac{21}{13}}(t)=-t^3+3t^2-3t+4-4t^{-1}+3t^{-2}-2t^{-3}+t^{-4}.$$

\par 
(2) Let us consider the case $\alpha =\frac{9}{16}$. 
Since 
$\frac{16}{9}=[2,5,2]^-$ and 
$\alpha _2=\frac{1}{[2, 5]^-}=\frac{5}{9},\ 
\alpha _1=\frac{1}{[2]^-}=\dfrac{1}{2}$, 
we will apply the formula of Theorem~\ref{2-11}(3)(vi). 
Then we have 
$\mathrm{wr}\bigl( \frac{9}{16}\bigr) 
=\mathrm{wr}\bigl( \frac{1}{2}\bigr) -(-1)^{z_1(\alpha_1)}=3$. 
Since $l=3,\ l^{\prime}=4$ and 
$$J_{\frac{16}{9}}(q)=1+2q+2q^2+3q^3+3q^4+3q^5+q^6+q^7,$$
we have 
\begin{align*}
V_{\frac{16}{9}}(t)
&=(-A^{-3})^{-2}(t^4-2t^3+2t^2-3t+3-3t^{-1}+t^{-2}-t^{-3}) \\ 
&=t^{\frac{5}{2}}-2t^{\frac{3}{2}}+2t^{\frac{1}{2}}-3t^{-\frac{1}{2}}+3t^{-\frac{3}{2}}-3t^{-\frac{5}{2}}+t^{-\frac{7}{2}}-t^{-\frac{9}{2}}. 
\end{align*}
\end{exam} 

\par \medskip 
\section{$q$-deformed integers derived from pairs of coprime integers}
\par 
In this section we introduce $q$-deformed integers derived from pairs of coprime integers. 
This concept is motivated by results on $q$-deformed rational numbers due to Morier-Genoud and Ovsienko. 
\par 
For a positive integer $a$ we set 
\begin{equation}
[a]_q:=1+q+q^2+\cdots +q^{a-1}\in \mathbb{Z}[q], 
\end{equation}
and $[0]_q:=0$. In the quotient field $Q(\mathbb{Z}[q])$ the equation $[a]_q=\frac{1-q^a}{1-q}$ holds. 
\par 
Morier-Genoud and Ovsienko~\cite{M-GO2} introduced $q$-deformations of rational numbers by using continued fractions. 
Let $c_1, \ldots , c_l$ be finite sequence of integers which are greater than or equal to $2$. Then 
$[c_1, \ldots, c_l]_q^-\in Q(\mathbb{Z}[q])$ is defined by 
\begin{equation}\label{eq4}
\begin{aligned}
{[c_1]_q}-\dfrac{q^{c_1-1}}{\vbox to 18pt{ }[c_2]_q-\dfrac{q^{c_2-1}}{\raisebox{0.5cm}{$\ddots$} \vbox to 25pt{ }-\dfrac{q^{c_{l-2}-1}}{\vbox to 18pt{ }[c_{l-1}]_q-\dfrac{q^{c_{l-1}-1}}{[c_l]_q}}}}. 
\end{aligned}
\end{equation}
For an $\alpha \in(1, \infty )\cap \mathbb{Q}$, by expanding as $\alpha =[c_1, \ldots, c_l]^-$, we set 
\begin{equation}
[\alpha ]_q:=[c_1, \ldots, c_l]_q^-, 
\end{equation}
and call it the \textit{$q$-rational number} of $\alpha $. 
By reducing \eqref{eq4} without multiplying elements of $\mathbb{Z}[q]$, 
$[\alpha ]_q$ is uniquely represented by the form 
\begin{equation}
[\alpha ]_q=\dfrac{N_q(\alpha )}{D_q(\alpha )}
\end{equation}
for some $N_q(\alpha ), D_q(\alpha )\in \mathbb{Z}[q]$, where 
\begin{enumerate}
\item[(i)] $N_q(\alpha ), D_q(\alpha )$ are coprime in $\mathbb{Z}[q]$, and 
\item[(ii)] if $\alpha =\frac{r}{s}$ with $(r,s)=1,\ r,s \geq 1$, then $N_1(\alpha )=r,\  D_1(\alpha )=s$. 
\end{enumerate}

For convenience we set  
$N_q(\frac{1}{0})=1,\ D_q(\frac{1}{0})=0$. 

\begin{thm}[{\bf Morier-Genoud and Ovsienko~\cite{M-GO2}}]\label{3-1}
Let $\alpha , \beta \geq 1$ be Farey neighbors, and express the Farey sum as $\alpha \sharp \beta =[c_1, \ldots , c_l]^-$. Then 
\begin{align*}
N_q(\alpha \sharp \beta )&=N_q(\alpha )+q^{c_l-1}N_q(\beta ), \\ 
D_q(\alpha \sharp \beta )&=D_q(\alpha )+q^{c_l-1}D_q(\beta ). 
\end{align*}
So, we define $[\alpha ]_q\oplus  [\beta ]_q$ by 
\begin{equation}\label{eq5}
[\alpha ]_q\oplus  [\beta ]_q := \dfrac{N_q(\alpha )+q^{c_l-1}N_q(\beta )}{D_q(\alpha )+q^{c_l-1}D_q(\beta )}. 
\end{equation}
Then 
\begin{equation}
[\alpha ]_q\oplus [\beta ]_q =[\alpha \sharp \beta ]_q. 
\end{equation}
\eqref{eq5} is said to be the weighted Farey sum of $[\alpha ]_q$ and $ [\beta ]_q$. 
\end{thm} 

\begin{rem}\label{3-2}
The integer $c_l$ in the above theorem is given by 
\begin{equation}\label{eq3-7}
c_l-1=\left\lceil \frac{N(\alpha )}{N(\beta )} \right\rceil , 
\end{equation}
where $N(\alpha ), N(\beta )$ are the numerators of $\alpha , \beta $, respectively. 
In fact, if we write $\alpha =\frac{p}{q},\ \beta =\frac{r}{s}$, then $\texttt{r}(\alpha \sharp \beta )=\frac{p+r}{r}=\frac{p}{r}+1$, where 
$\texttt{r}$ is the operator defined in \cite[Lemma 1.3]{Kogiso-Wakui_OJM}. 
On the other hand, it also given by $\texttt{r}(\alpha \sharp \beta )=[c_l, \ldots ,c_1]^-=c_l-\frac{1}{[c_{l-1}, \ldots , c_1]^-}$. 
Thus, $\frac{p}{r}+1=c_l-\frac{1}{[c_{l-1}, \ldots , c_1]^-}$. 
By comparing the integer parts of both sides of this equation, \eqref{eq3-7} is obtained. 
\end{rem} 

Let us introduce a convenient method to compute $[\alpha \sharp \beta ]_q$ by Euclidean algorithm. 
For a pair $(a, b)$ of positive and coprime integers we define $(a, b)_q$ by 
\begin{equation}\label{eq3-8}
(a, b)_q
:=\begin{cases}
(a-r, r)_q+q(a, b-a)_q & \text{if $a<b$},\\ 
(a-b, b)_q+q^{\lceil \frac{a}{b}\rceil }(r, b-r)_q & \text{if $a>b$}, 
\end{cases}
\end{equation}
where $r$ is the remainder when $b$ is divided by $a$ in case where $a<b$, and 
when $a$ is divided by $b$ in case where $a>b$, 
and also $(1, n)_q=(n, 1)_q=[1+n]_q$ for any non-negative integer $n$. 

\begin{thm}\label{3-3}
If $\alpha =\frac{x}{a},\ \beta =\frac{y}{b}\geq 1$ are Farey neighbors, then 
\begin{equation}\label{eq3-9}
D_q(\alpha \sharp \beta )=(a, b)_q,\qquad N_q(\alpha \sharp \beta )=(x, y)_q. 
\end{equation} 
Thus 
\begin{equation}\label{eq3-10}
[\alpha \sharp \beta ]_q=\dfrac{(x,y)_q}{(a, b)_q}. 
\end{equation}
\end{thm}
\begin{proof}
The theorem can be verified by induction on the number of times of the operation $\sharp $. 
\par 
I. If $a+b=1$, then $\alpha =\frac{a}{1},\ \beta =\frac{1}{0}$, and \eqref{eq3-10} can be easily verified. 
If $a+b=2$, then $\alpha =\frac{a}{1},\ \beta =\frac{a+1}{1}$. 
Since $\alpha \sharp \beta =[a+1, 2]^-$, we have 
$[\alpha \sharp \beta ]_q=\frac{[a+1]_q[2]_q-q^a}{[2]_q}=\frac{[a]_q+q[a+1]_q}{[2]_q}=\frac{(a, a+1)_q}{(1,1)_q}$. 

\par 
I\kern-0.1em I. Let $\alpha , \beta \geq 1$ be Farey neighbors.  
Under the assumption $a+b\geq 3$, it is enough to show that 
if $\alpha , \beta $ satisfy \eqref{eq3-9}, then $\alpha \sharp \beta $ so does. 
Under this assumption we note that $a<b$ if and only if $x<y$. 
\par 
Let us consider the case where $a<b$, and therefore $x<y$. 
Let $r$ be the remainder when $b$ is divided by $a$, and 
$s$ be the remainder when $y$ is divided by $x$. 
Then, the pair 
$\frac{x}{a}, \frac{y-x}{b-a}$ and the pair 
 $\frac{x-s}{a-r}, \frac{s}{r}$ are Farey neighbors, respectively, and 
$\frac{y}{b}=\frac{x}{a}\sharp\frac{y-x}{b-a}, \ 
\frac{x}{a}=\frac{x-s}{a-r}\sharp\frac{s}{r}$. 
By induction hypothesis, 
$$D_q(\alpha )=(a-r, r)_q,\ N_q(\alpha )=(x-s, s)_q,\ 
D_q(\beta )=(a, b-a)_q,\ N_q(\beta )=(x, y-x)_q.$$
Since $\lceil \frac{x}{y} \rceil =1$, by Theorem~\ref{3-1} and \eqref{eq3-8} we have 
\begin{align*}
D_q(\alpha \sharp \beta )
&=(a-r, r)_q+q(a, b-a)_q=(a, b)_q,\\ 
N_q(\alpha \sharp \beta )
&=(x-s, s)_q+q(x, y-x)_q=(x, y)_q. 
\end{align*}

Next, let us consider the case where $a>b$, and therefore $x>y$. 
Let $r$ be the remainder when $a$ is divided by $b$, and 
$s$ be the remainder when $x$ is divided by $y$. 
Then, the pair 
$\frac{x-y}{a-b}, \frac{y}{b}$ and the pair $\frac{s}{r}, \frac{y-s}{b-r}$ are Farey neighbors, respectively, and 
$\frac{x}{a}=\frac{x-y}{a-b}\sharp\frac{y}{b}, \ 
\frac{y}{b}=\frac{s}{r}\sharp\frac{y-s}{b-r}$. 
By the same argument above, we see that 
$D_q(\alpha \sharp \beta )=(a, b)_q,\  
N_q(\alpha \sharp \beta )=(x, y)_q$. 
Moreover, 
since $x$ can be written as $x=my+s$ for some $0<s<y$ and $1\leq a-mb\leq b$, 
it follows that $\lceil \frac{x}{y} \rceil =m+1=\lceil \frac{a}{b} \rceil $. 
This implies that \eqref{eq3-9} holds for $\alpha \sharp \beta$. 
\end{proof} 

We use the following convention: For a positive integer $n$ 
\begin{equation}
(n)_q:=[n]_q\ (=1+q+\cdots +q^{n-1}),\quad (0)_q:=0.
\end{equation}

\begin{exam}
By the formula \eqref{eq3-10}, 
$\bigl[ \frac{17}{5}\bigr]_q=\frac{(10,7)_q}{(3,2)_q}=\frac{(3,7)_q+q^2(3,4)_q}{(3)_q+q^2(2)_q}$. 
Since $(3,4)_q=(3)_q+q(4)_q,\ 
(3,7)_q=(3)_q+q(3,4)_q$, we have  
$$
\Bigl[ \dfrac{17}{5}\Bigr]_q
=\dfrac{(1+q+q^2)(3)_q+(1+q)q^2(4)_q}{(3)_q+q^2(2)_q} 
=\dfrac{(3)_q^2+q^2(2)_q(4)_q}{(3)_q+q^2(2)_q}. $$
\end{exam} 

In the sequel, we will investigate properties of $(a, b)_q$. 

\begin{lem}\label{3-5}
Let $(a, x)$ be a pair of coprime integers with $1\leq a\leq x$, and write in the form 
$x=ca+r\ (0\leq r<a)$. Then 
\begin{align}
(a,x)_q&=(a-r, r)_q(c)_q+q^c(a, r)_q, \label{eq3-12}\\  
\mathrm{deg}(x-a, a)_q&=\mathrm{deg}(a, x)_q-1, \label{eq3-13} \\ 
\mathrm{deg}(a-r, r)_q&=\mathrm{deg}(a, x)_q-\left\lceil \dfrac{x}{a}\right\rceil , \label{eq3-14}\\
\mathrm{deg}(a, x)_q&=\mathrm{deg}(x, a)_q, \label{eq3-15}\\ 
(x, a)_q&=q^{\mathrm{deg}(x, a)_q}(a, x)_{q^{-1}}. \label{eq3-16}
\end{align}
\end{lem}
\begin{proof}
The equation \eqref{eq3-12} can be directly derived from the definition of $(a,x)_q$. 
\par 
We will show the equations \eqref{eq3-13} and \eqref{eq3-15}. 
To do this, let $k\geq 2$ be an integer, and 
by induction on $k$ we show that 
\begin{equation*}
\mathrm{deg}(x, a)_q=\mathrm{deg}(a, x)_q=\mathrm{deg}(x-a, a)_q+1 \tag*{$(\ast )$}
\end{equation*}
for all pairs $(a, x)$ of coprime integers satisfying $1\leq a\leq x$ and $a+x=k$. 
\par 
When $k=2$, we have $a=x=1$, and when $k=3$, we have $a=1,\ x=2$. 
In the both cases thus the equation $(\ast )$ holds. 
\par 
Assume that $k>2$ and the equation $(\ast )$ holds for all pairs $(a^{\prime}, x^{\prime})$ of coprime integers satisfying $1\leq a^{\prime}\leq x^{\prime},\ a^{\prime}+x^{\prime}<k$. 
Let $(a, x)$ be a pair of coprime integers satisfying $1\leq a\leq x,\ a+x=k$. 
We express $x$ as the form $x=ca+r\ (0\leq r<a)$. 
Since $\text{deg}(a, r)_q=\text{deg}(r, a)_q=\text{deg}(a-r, r)_q+1$
by induction hypothesis, 
we have 
$\text{deg}\bigl( (a-r, r)_q(c)_q\bigr) 
=c+d-2,\ 
\text{deg}\bigl( q^c(a, r)_q \bigr) 
=c+d
$ by setting $d:=\text{deg}(a, r)_q$. 
Thus 
$\text{deg}(a, x)_q=c+d=\text{deg}(a-r, r)_q+1+c$, 
and in particular, 
$\text{deg}(a, x)_q=\text{deg}(a, x-a)_q+1$
by combining it with 
$\text{deg}(a, x)_q>\text{deg}(a-r, r)_q$ and $(a, x)_q=(a-r, r)_q+q(a, x-a)_q$. 
Since 
$(x, a)_q=(x-a, a)_q+q^{\lceil \frac{x}{a}\rceil } (r, a-r)_q$, 
we have 
\begin{align*}
\text{deg}(x-a, a)_q
=\text{deg}(a, x-a)_q
&=\text{deg}(a, x)_q-1\\ 
&=\text{deg}(a-r, r)_q+c 
=\text{deg}(r, a-r)_q+c. 
\end{align*}

\par 
If $a>1$, then $\lceil \frac{x}{a}\rceil =c+1$. It follows that 
$$\text{deg}(x-a, a)_q<\text{deg}(r, a-r)_q+c+1=\text{deg}(r, a-r)_q+\left\lceil \frac{x}{a}\right\rceil .$$
Therefore, we obtain
$\text{deg}(x, a)_q
=\left\lceil \frac{x}{a}\right\rceil +\text{deg}(r, a-r)_q 
=c+d
=\text{deg}(a, x)_q$. 
\par 
If $a=1$, then $(x,a)_q=(x+1)_q=(a, x)_q$, and  hence 
$\text{deg}(x, a)_q=\text{deg}(a, x)_q$. 
\par 
This completes the induction argument, 
and therefore the equation $(\ast )$ holds for all pairs $(a, x)$ of coprime integers with $1\leq a\leq x$. 

\par 
The equation \eqref{eq3-14} can be obtained from 
$(x, a)_q=(x-a, a)_q+q^{\lceil \frac{x}{a}\rceil } (r, a-r)_q$ 
and 
$\text{deg}(a, x)_q=\text{deg}(x, a)_q>\text{deg}(x-a, a)_q$. 

\par 
Finally, we will derive the equation \eqref{eq3-16}. 
To do this, let $k \geq 2$ be an integer, and 
we show that by induction on $k$ the two equations \eqref{eq3-16} and 
\begin{equation}\label{eq3-17}
(a, x)_q=q^{\mathrm{deg}(a, x)_q}(x, a)_{q^{-1}}
\end{equation}
hold for all pairs $(a, x)$ of coprime integers satisfying $1\leq a\leq x,\ x+a\leq k$ at the same time. 

\par 
I. When $a=x=1$, it can be immediately verified that the desired equations hold. 
\par 
I\kern-0.1em I. 
Let $k\geq 2$ be an integer, and assume that \eqref{eq3-16} and \eqref{eq3-17} hold for all coprime pairs $(a, x)$ of integers satisfying $1\leq a\leq x,\ x+a\leq k$. 
Let $(a, x)$ be a pair of coprime integers satisfying $1\leq a\leq x,\ a+x=k+1$, and $r$ be the remainder of $x$ divided by $a$. 
Since 
$(a-r)+r=a<a+x,\ a+(x-a)=x<a+x$, by induction hypothesis we have 
$(a, x)_q=(a-r, r)_q+q(a, x-a)_q
=q^{\text{deg}(a-r, r)_q}(r, a-r)_{q^{-1}}+q^{1+\text{deg}(x-a, a)_q}(x-a, a)_{q^{-1}}$. 
By \eqref{eq3-13} and \eqref{eq3-14} the right-hand side is rewritten as 
$q^{\text{deg}(a, x)_q}\bigl( q^{-\left\lceil \frac{x}{a}\right\rceil }(r, a-r)_{q^{-1}}+(x-a, a)_{q^{-1}} \bigr) 
=q^{\text{deg}(a, x)_q}(x, a)_{q^{-1}}$. 
Replacing $q$ in the equation with $q^{-1}$ we have $(a, x)_{q^{-1}}=q^{-\text{deg}(a, x)_q}(x, a)_q$. 
This is equivalent to \eqref{eq3-16}. 
This completes the induction argument. 
\end{proof} 

\begin{lem}\label{3-6}
Let $(a, x)$, $(b, y)$ be pairs of coprime integers such that $1\leq a\leq x$, $1\leq b\leq y$ and $\frac{x}{a}, \frac{y}{b}$ are Farey neighbors. 
Then
\begin{align}
\mathrm{deg}(a+b, x+y)_q
&=\mathrm{deg}(b, y)_q+\left\lceil \dfrac{x}{y} \right\rceil , \label{eq3-18}\\ 
\mathrm{deg}(a+b, x+y)_q
&=\mathrm{deg}(a, x)_q+\left\lfloor \dfrac{b}{a}+1 \right\rfloor . \label{eq3-19}
\end{align}
\end{lem}
\begin{proof}
This lemma can be proved by induction on $k\geq 3$ when we set  $x+y=k$. 

\par 
I. Let $k=3$. Then 
$(x, y)=(1, 2)$ and $(a, b)=(1, 1)$. 
Thus \eqref{eq3-18}, \eqref{eq3-19} become 
$\text{deg}(2, 3)_q=\text{deg}(1, 2)_q+1, \  
\text{deg}(2, 3)_q=\text{deg}(1, 1)_q+2$, respectively. 
It can be easily see that these equations hold. 

\par 
I\kern-0.1em I. Let $k\geq 3$. 
Assume that \eqref{eq3-18}, \eqref{eq3-19} hold 
for all positive integers $a^{\prime}, b^{\prime}, x^{\prime}, y^{\prime}$ such that 
$x^{\prime}+y^{\prime}<k$,  $a^{\prime}y^{\prime}-b^{\prime}x^{\prime}=1,\ 1\leq a^{\prime}\leq x^{\prime},\ 1\leq b^{\prime}\leq y^{\prime}$. 
Let us consider positive integers $a, b, x, y$ which satisfy $x+y=k$ and $ay-bx=1,\ 1\leq a\leq x,\ 1\leq b\leq y$, and let $r$ be the remainder of $
x+y$ divided by $a+b$. 
Since 
$a+b<x+y$, we have 
$$(a+b, x+y)_q=(a+b-r, r)_q+q(a+b, (x+y)-(a+b))_q.$$
By \eqref{eq3-14} we have 
\begin{equation*}
\text{deg}(a+b, x+y)_q=1+\text{deg}(a+b, (x+y)-(a+b))_q. \tag*{$(\ast )$}
\end{equation*}

\par 
(1) In the case where $a<x-a$ and $b<y-b$, 
by induction hypothesis 
\begin{align*}
\text{deg}(a+b, (x+y)-(a+b))_q
&=\text{deg}(a+b, (x-a)+(y-b))_q\\ 
&=\text{deg}(b, y-b)_q+\left\lceil \frac{x-a}{y-b} \right\rceil \\ 
&=\text{deg}(a, x-a)_q+\left\lfloor \frac{b}{a}+1 \right\rfloor  . 
\end{align*}
Applying \eqref{eq3-13} and  Lemma~\ref{2-2} we have \eqref{eq3-18}, \eqref{eq3-19}. 

\par 
(2) The case where $a<x-a$ and $b\geq y-b$ does not occur since 
$2b-y
=2b-\frac{xb+1}{a}
=\frac{(2a-x)b-1}{a}<0$, which is a contradiction. 
Similarly, the case where $a\geq x-a$ and $b<y-b$ does not occur. 
\par 
(3) In the case where $a\geq x-a$ and $b\geq y-b$, 
we see that $a>x-a$ and $b>y-b$. 
\par 
(i) If $x-a>0$ and $y-b>0$, then one can apply induction hypothesis to $\text{deg}((y-b)+(x-a), b+a)_q$ since $a+b<x+y=k$. Then 
we have 
\begin{align*}
\text{deg}(a+b, (x-a)+(y-b))_q
&=\text{deg}((y-b)+(x-a), b+a)_q\\ 
&=\text{deg}(x-a, a)_q+\left\lceil \frac{b}{a} \right\rceil \\ 
&=\text{deg}(y-b, b)_q+\left\lfloor \frac{x-a}{y-b} +1\right\rfloor . 
\end{align*}
Using \eqref{eq3-13} and $(\ast )$ we have 
$\text{deg}(a+b, x+y)_q=\text{deg}(a, x)_q+\lceil \frac{b}{a} \rceil $. 
\par 
We see that $a>1$ and $\frac{b}{a}\not\in \mathbb{Z}$. 
So, $\lceil \frac{b}{a} \rceil =\lfloor \frac{b}{a} +1\rfloor $, and therefore, 
$\text{deg}(a+b, x+y)_q=\text{deg}(a, x)_q+\left\lfloor \frac{b}{a} +1\right\rfloor $. 
Similarly, by \eqref{eq3-13}, Lemma~\ref{2-2} and $(\ast )$, we have 
\begin{align*}
\text{deg}(a+b, x+y)_q
&=1+\text{deg}(y-b, b)_q +\left\lfloor \frac{x-a}{y-b} +1\right\rfloor \\ 
&=\text{deg}(b, y)_q+\left\lfloor \frac{x-a}{y-b} +1\right\rfloor . 
\end{align*}

\par 
$\bullet$ If $y-b>1$, then $x-a, y-b$ are coprime, and hence 
$\frac{x-a}{y-b}\not\in \mathbb{Z}$. It follows that  
$\lfloor \frac{x-a}{y-b} +1\rfloor =\lceil \frac{x-a}{y-b} \rceil =\lceil \frac{x}{y} \rceil $ by Lemma~\ref{2-2}, and 
hence $\text{deg}(a+b, x+y)_q=\text{deg}(b, y)_q+\left\lceil \frac{x}{y} \right\rceil $. 
\par 
$\bullet$ If $y-b=1$, then 
$\lfloor \frac{x-a}{y-b} +1\rfloor =\lceil \frac{x-a}{y-b} \rceil +1=\lceil \frac{x}{y} \rceil $ by Lemma~\ref{2-2}, 
and hence 
$\text{deg}(a+b, x+y)_q=\text{deg}(b, y)_q+\lceil \frac{x}{y} \rceil $. 

\par 
(ii)
If $y-b=0$, then $y=b=1$. 
However, this is a contradiction since the Farey neighbors $\frac{x}{a}, \frac{y}{b}=\frac{1}{1}$ should be greater than or equal to $1$. 

\par 
(iii)
If $x-a=0$, then $x=a=1$ and $y=b+1$. 
In this case we can directly prove \eqref{eq3-18}, \eqref{eq3-19} by easy computation. 
\end{proof} 

\begin{thm}\label{3-7}
Let $(a, x)$, $(b, y)$ be pairs of coprime integers such that $1\leq a\leq x$, $1\leq b\leq y$ and $\frac{x}{a}, \frac{y}{b}$ are Farey neighbors. 
Then
\begin{align}
(a+b, x+y)_q&=(a, x)_q+q^{\lceil \frac{x}{y} \rceil }(b, y)_q, \label{eq3-20} \\ 
(y+x, b+a)_q&=(y, b)_q+q^{\lfloor \frac{b}{a} +1\rfloor }(x, a)_q. \label{eq3-21}
\end{align}
\end{thm}
\begin{proof}
\par 
First, we will show that the equation \eqref{eq3-20} holds. 
If $a+b=2$, that is, $a=b=1$, then $y=x+1$. 
In this case, \eqref{eq3-20} is written as 
$(2, 2x+1)_q=(1, x)_q+q(1, x+1)_q$. 
This equation can be obtained by direct computation. 
\par 
Next we consider the case $a+b\geq 3$. 
We note that $a\leq x,\ b<y$ by assumption of the theorem and $a+b\geq 3$. 
\par 
(1) Let us consider the case $x<y$. 
We prove the equation 
\begin{equation}\label{eq3-22}
\dfrac{(a+b, x+y)_q}{(a, x)_q}=1+q^{\left\lceil \frac{x}{y} \right\rceil }\dfrac{(b, y)_q}{(a, x)_q}, 
\end{equation}
which is equivalent to \eqref{eq3-20}. 
Since $\lceil \frac{x}{y} \rceil =1$ and $b<y$, we have 
$$
(\text{R.H.S. of \eqref{eq3-22}})
=1+q\Bigl[ \dfrac{b}{a}\sharp \dfrac{y}{x}\Bigr]_q 
=\dfrac{D_q\bigl( \frac{b}{a}\bigr) +qD_q\bigl( \frac{y}{x}\bigr) +qN_q\bigl( \frac{b}{a}\bigr) +q^2N_q\bigl( \frac{y}{x}\bigr) }{D_q\bigl( \frac{b}{a}\bigr) +qD_q\bigl( \frac{y}{x}\bigr) }. 
$$
\noindent 
On the other hand, since $\lceil \frac{a+b}{x+y} \rceil =1$ from $a\leq x,\ b<y$, it follows that 
\begin{align*}
(\text{L.H.S. of \eqref{eq3-22}})
&=\dfrac{D_q\bigl( \frac{a}{b}\bigr)+qN_q\bigl( \frac{a}{b}\bigr) +q^{\left\lceil \frac{a+b}{x+y} \right\rceil }\bigl( D_q\bigl( \frac{y}{x}\bigr)+qN_q\bigl( \frac{y}{x}\bigr) \bigr)}{D_q\bigl( \frac{b}{a}\bigr)+q^{\left\lceil \frac{a+b}{x+y} \right\rceil }D_q\bigl( \frac{y}{x}\bigr)}  \\ 
&=\dfrac{D_q\bigl( \frac{a}{b}\bigr)+qN_q\bigl( \frac{a}{b}\bigr) +q\bigl( D_q\bigl( \frac{y}{x}\bigr)+qN_q\bigl( \frac{y}{x}\bigr)\bigr) }{D_q\bigl( \frac{b}{a}\bigr)+qD_q\bigl( \frac{y}{x}\bigr)} \\ 
&=(\text{R.H.S. of \eqref{eq3-22}}). 
\end{align*}

\par 
(2) Let us consider the case $x>y$. 
The equation \eqref{eq3-20} is rewritten as 
\begin{equation}\label{eq3-23}
\dfrac{(a+b, x+y)_q}{(b, y)_q}=\dfrac{(a, x)_q}{(b, y)_q}+q^{\left\lceil \frac{x}{y} \right\rceil }. 
\end{equation}
To prove the equation, 
applying \eqref{eq3-16} and the result of (1), we have 
$$
\dfrac{(a+b, x+y)_q}{(b, y)_q}
=q^{\text{deg} (x+y, a+b)_q-\text{deg} (y, b)_q} \Bigl( 1+q^{-1}\dfrac{(x, a)_{q^{-1}}}{(y, b)_{q^{-1}}}\Bigr) , 
$$
and hence 
$$
(\text{L.H.S. of \eqref{eq3-23}})
=q^{\text{deg} (x+y, a+b)_q-\text{deg} (y, b)_q} +q^{\text{deg} (x+y, a+b)_q-\text{deg} (x, a)_q-1}\dfrac{(a, x)_q}{(b, y)_q}.$$
By Lemma~\ref{3-6} and \eqref{eq3-16} the above value coincides with the right-hand side of \eqref{eq3-23}. 
\par 
The equation \eqref{eq3-21} can be obtained from \eqref{eq3-20} by using Lemma~\ref{3-5} as follows: 
\begin{align*}
(y+x, b+a)_q
&=q^{\text{deg}(a+b, x+y)_q}\bigl( (a, x)_{q^{-1}}+q^{-\left\lceil \frac{x}{y}\right\rceil }(b, y)_{q^{-1}}\bigr) \displaybreak[0]\\ 
&=q^{\left\lfloor \frac{b}{a} +1\right\rfloor }(x, a)_q+(y, b)_q. 
\end{align*}
\end{proof}

\begin{cor}
Let $(a, x)$, $(b, y)$ be pairs of coprime integers such that $\frac{x}{a}, \frac{y}{b}$ are Farey neighbors. 
We define $(a, x)_q\oplus (b, y)_q$ by 
\begin{equation}
(a, x)_q\oplus (b, y)_q:=
\begin{cases}
(a, x)_q+q^{\left\lceil \frac{x}{y}\right\rceil }(b, y)_q & \text{if $1\leq a\leq x$},\\ 
(a, x)_q+q^{\left\lfloor \frac{x}{y} +1\right\rfloor }(b, y)_q & \text{if $1\leq y\leq b$}. \ \ 
\end{cases}
\end{equation}
Then $(a, x)_q\oplus (b, y)_q=(a+b, x+y)_q$. 
\end{cor}
\begin{proof}
The result is a direct consequence of Theorem~\ref{3-7}. 
\end{proof}

\section{The normalized Jones polynomials based on $q$-Farey sums and its applications} 
By \cite[Proposition A.1]{M-GO2} and $[\alpha +1]_q=\frac{D_q(\alpha )+qN_q(\alpha )}{D_q(\alpha )}$ we see that the normalized Jones polynomial $J_{\alpha}(q)$ can be computed by the following formula: 

\begin{lem}\label{4-1}
For a rational number $\alpha >1$, 
\begin{equation}\label{eq4-1}
J_{\alpha }(q)=qN_q(\alpha )+(1-q)D_q(\alpha ). 
\end{equation} 
\end{lem} 

Furthermore we have: 

\begin{lem}\label{4-2}
The normalized Jones polynomials $J_{\alpha }(q)$ can be inductively computed from the initial data $J_1(q)=1,\ J_{\infty}(q)=q$ and 
the formula
\begin{equation}\label{eq4-2}
J_{\alpha \sharp \beta}(q)=J_{\alpha }(q)+q^{\left\lceil \frac{N(\alpha )}{N(\beta )}\right\rceil }J_{\beta}(q)
\end{equation}
for all Farey neighbours $\alpha , \beta>1$, where 
$N(\alpha ), N(\beta )$ are the numerators of $\alpha , \beta$, respectively. 
\end{lem} 
\begin{proof}
The equation \eqref{eq4-2} is obtained from combining Theorem~\ref{3-1}, Lemma~\ref{4-1} and the equation \eqref{eq3-7}. 
\end{proof} 

The normalized Jones polynomial $J_{\alpha}(q)$ can be computed from $(a, b)_q$ as follows. 

\begin{thm}\label{4-3}
Let $(a, x)$ be a pair of coprime integers with $1\leq a<x$. Then
\begin{align}
J_{\frac{x}{a}}(q)&=q^2(a, x-a)_q-(q-1)(a, x)_q, \label{eq4-3} \\ 
(a, x)_q&=J_{\frac{x}{a}}(q)+q(a-r, r)_q,  \label{eq4-4}
\end{align}
where $r$ is the remainder when  $x$ is divided by $a$. 
\end{thm}
\begin{proof}
Let us prove the equation \eqref{eq4-3}. 
It is proved by induction on the number of times of the operation $\sharp $. 
\par 
First of all, let us consider the case $a=1$. 
Dividing into the two cases where $b=1$ or not, we can easily see that the equation \eqref{eq4-3} holds. 
\par 
Next, let $\alpha =\frac{b}{a},\ \beta =\frac{d}{c}\geq 1$ be Farey neighbors, and 
assume that they satisfy \eqref{eq4-3}. 
Then we show that for $\alpha \sharp\beta =\frac{b+d}{a+c}$  the equation \eqref{eq4-1} holds by dividing into the four cases as follows. 
\par 
(1) In the case where $a=b$, we have $a=b=1$ and $c=1,\ d=2$. 
So, $\alpha \sharp\beta =\frac{3}{2}$ and hence the right-hand side of \eqref{eq4-3} is $q^2(2, 3-2)_q-(q-1)(2, 3)_q=1+q+q^3$. 
On the other hand, by $\frac{3}{2}=\frac{1}{1}\sharp \frac{2}{1}$ we have 
$J_{\frac{3}{2}}(q)=J_1(q)+q^{\lceil \frac{1}{2}\rceil } J_2(q)=1+q+q^3$. 
Thus, in this case \eqref{eq4-3} holds. 

\par 
(2) In the case where $a<b$ and $d-c>1$, by Lemmas~\ref{2-2}, \ref{4-2}, induction hypothesis and \eqref{eq3-16} we have 
\begin{align*}
J_{\alpha \sharp \beta }(q)
&=q^2(a, b-a)_q-(q-1)(a, b)_q+q^{\left\lceil \frac{b}{d}\right\rceil}\bigl( q^2(c, d-c)_q-(q-1)(c, d)_q\bigr) \displaybreak[0]\\
&=q^2\bigl( (a, b-a)_q +q^{\left\lceil \frac{b-a}{d-c}\right\rceil} (c, d-c)_q\bigr) \\ 
&\qquad -(q-1)\bigl( (a, b)_q+q^{\left\lceil \frac{b}{d}\right\rceil} (c, d)_q \bigr)  \quad (\frac{b}{d}=\frac{b-a}{d-c}\sharp \frac{a}{c},\ d-c>1)\displaybreak[0]\\
&=q^2(a+c, (b-a)+(d-c))_q -(q-1)(a+c, b+d)_q. 
\end{align*}
It follows that \eqref{eq4-3} holds for $\alpha \sharp \beta $. 
\par 
(3) In the case where $a<b$ and $d-c=1$, as the same argument of (2), one can verify the equation \eqref{eq4-3} for $\alpha \sharp \beta $. 
\par 
(4) In the case where $a<b$ and $d=c$, we have $c=d=1$. 
This leads to $a=b+1$ by $ad-bc=1$, which is a contradiction for $a\leq b$. Thus, this case does not occur. 
\par 
This completes the induction argument. 
\par 
The equation \eqref{eq4-4} can be derived from \eqref{eq4-3} as follows. 
For the case $a=b=1$, the desired equation can be easily verified. 
So, we consider the case of $a<b$. 
By definition of $(a, b)_q$, we have 
$(a, b)_q=(a-r, r)_q+q(a, b-a)_q$. 
This implies that $
J_{\frac{b}{a}}(q)+q(a-r, r)_q
=J_{\frac{b}{a}}(q)+q\bigl( (a, b)_q-q(a, b-a)_q\bigr) =(a,b)_q$. 
\end{proof} 

As an application of Theorem~\ref{4-3} we have: 

\begin{thm}\label{thm23}
For a pair $(a, x)$ of coprime integers with $1\leq a\leq x$, 
\begin{align}
N_q\Bigl( \dfrac{x}{a}\Bigr) &=(a, x-a)_q, \label{eq4-5}\\ 
D_q\Bigl( \dfrac{x}{a}\Bigr) &=(a-r, r)_q  \label{eq4-6}
\end{align}
where $r$ is the remainder when  $x$ is divided by $a$. 
\end{thm}
\begin{proof}
The equation \eqref{eq4-5} is proved by induction of the number of times of $\sharp$ as follows. 
In the case of $a=1$ the equation \eqref{eq4-5} holds since 
$N_q( \frac{n}{1}) =[n]_q=(n)_q=(1, n-1)_q$. 
Let $\frac{x}{a}, \frac{y}{b}$ be Farey neighbors consisting of rational numbers greater than $1$, and assume that they satisfy 
$N_q( \frac{x}{a}) =(a, x-a)_q,\ 
N_q( \frac{y}{b}) =(b, y-b)_q$. 
Then 
\begin{align*}
(a+b, x+y-(a+b))_q
&=(a, x-a)_q+q^{\left\lceil \frac{x}{y}\right\rceil }(b, y-b)_q \quad (\text{Theorem~\ref{3-7}}) \\ 
&=N_q\Bigl( \dfrac{x}{a}\Bigr) +q^{\left\lceil \frac{x}{y}\right\rceil }N_q\Bigl( \dfrac{y}{b}\Bigr)  \quad (\text{induction hypothesis}) \\ 
&=N_q\Bigl( \dfrac{x}{a}\sharp \dfrac{y}{b}\Bigr)  \quad (\text{Theorem~\ref{3-1} and \eqref{eq3-7}}) \\ 
&=N_q\Bigl( \dfrac{x+y}{a+b}\Bigr) , 
\end{align*}
and hence the equation \eqref{eq4-5} holds for the Farey sum $\frac{x}{a}\sharp \frac{y}{b}$. 
This completes the induction argument. 
\par 
The equation \eqref{eq4-6} can be reduced from \eqref{eq4-5} as follows. 
By Lemma~\ref{4-1} and \eqref{eq4-4} the following equation holds: 
\begin{equation}\label{eq4-7}
(a, p)_q=qN_q\Bigl( \dfrac{p}{a}\Bigr) +(1-q)D_q\Bigl( \dfrac{p}{a}\Bigr) +q(a-r, r)_q. 
\end{equation}
By \eqref{eq4-5} this implies that 
\begin{align*}
(1-q)D_q\Bigl(\dfrac{p}{a}\Bigr) 
&=(a, p)_q-q(a-r, r)_q-q(a, p-a)_q \\ 
&=(1-q)(a-r, r)_q, 
\end{align*}
and  hence we have \eqref{eq4-6}. 
\end{proof} 

If positive integers $a, x$ are coprime, then so are  $a, a+x$. 
By applying Theorem~\ref{thm23} we have: 

\begin{cor}\label{cor24}
For a pair $(a, x)$ of coprime integers with $1\leq a\leq x$, 
\begin{equation}
(a, x)_q=N_q\Bigl( \dfrac{a+x}{a}\Bigr). 
\end{equation}
\end{cor}

\par \medskip 
\begin{rem}
By Corollary~\ref{cor24} and Morier-Genoud and Ovsienko's results~\cite{M-GO2} we see that coefficients of 
$(a, x)_q$ have a quiver theoretical meaning as follows. 
Let $G$ be an oriented graph, and $V(G)$ denote the set of vertices of $G$. 
A subset $C\subset V(G)$ is an $\ell$-vertex closure if $\sharp C=\ell$ and there are no edges   
from vertices in $C$  to vertices in $V(G)-C$. 
A vertex in a $1$-vertex closure is called a sink. 
For a pair $(a, x)$ of coprime integers with $1\leq a\leq x$, we write in the form 
$\frac{x}{a}=[a_1, a_2, \ldots , a_{2m-1}, a_{2m}]$. 
Consider the quiver 
\begin{equation}
\tilde{G}\Bigl( \dfrac{x}{a} \Bigr)\ \mathrm{ :\ \raisebox{-0.2cm}{\includegraphics[height=1.1cm]{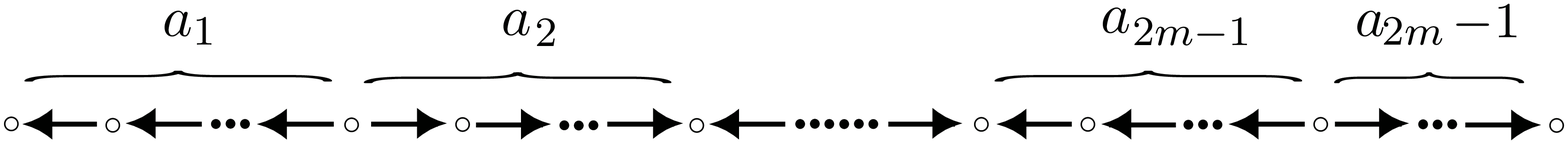}}}
\end{equation}
By \cite{M-GO2} then $(a, x)_q$ can be expressed as 
\begin{equation}
(a, x)_q=1+\rho_1q+\rho_2q^2+\cdots +\rho_{n-1}q^{n-1}+q^n, 
\end{equation}
where $n=a_1+a_2+\cdots +a_{2m}$, and 
\begin{equation}
\rho_i=\sharp \Bigl\{\ \text{the $i$-vertex closures of $\tilde{G}\Bigl( \dfrac{x}{a} \Bigr)$} \Bigr\}
\end{equation}
for each $1\leq i\leq n-1$. 
\end{rem}

\par \medskip \noindent 
{Acknowledgments}. 
I would like to thank Professors Takeyoshi Kogiso and Eiko Kin for helpful discussions and comments. 
I would like to thank the organizers of the conference \lq\lq Low dimensional topology and number theory XIII" to the memory of Professor Toshie Takata for inviting me in the conference. 
I had studied together with her since postgraduate degree. 
I would like to dedicate this paper in memory of Toshie Takata san with my heartfelt thanks.

\end{document}